\documentclass[12pt,a4paper]{amsart}
\usepackage{amsfonts}
\usepackage{mathrsfs}
\usepackage{amsthm,amsxtra}
\usepackage{amssymb}
\usepackage{amsmath}
\usepackage{amscd}
\usepackage{hyperref}
\hypersetup{
    colorlinks=true,
    linkcolor=blue,
    citecolor=red,
}
\usepackage[latin2]{inputenc}
\usepackage{t1enc}
\usepackage[mathscr]{eucal}
\usepackage{indentfirst}
\usepackage{graphicx}
\graphicspath{ {images/} }
\usepackage{subfig}
\usepackage{enumerate}
\usepackage{graphics}
\usepackage{pict2e}
\usepackage{epic}
\numberwithin{equation}{section}
\usepackage[margin=2.9cm]{geometry}
\usepackage{epstopdf} 
\usepackage{longtable}
\usepackage[inline]{enumitem}
\usepackage{tikz}
\usepackage{tikz-cd}

\theoremstyle{plain}
\newtheorem{theorem}{Theorem}[section]
\newtheorem{proposition}[theorem]{Proposition}
\newtheorem{lemma}[theorem]{Lemma}
\newtheorem{corollary}[theorem]{Corollary}

 \theoremstyle{definition}
\newtheorem{definition}[theorem]{Definition}

\theoremstyle{remark}
\newtheorem{remark}[theorem]{Remark}

\begin{document}

\title{Powers of Dehn Twists generating Right-Angled Artin Groups}

\author{Donggyun Seo}

\address{Seoul National University \\ Department of Mathematical Sciences \\
Seoul, Korea} 

\email{seodonggyun@snu.ac.kr}



\begin{abstract}
We give a bound for the exponents of powers of Dehn twists to generate a right-angled Artin group.
Precisely, if $\mathcal{F}$ is a finite collection of pairwise distinct simple closed curves on a finite type surface and if $N$ denotes the maximum of the intersection numbers of all pairs of curves in $\mathcal{F}$, then we prove that $\{T_\gamma^n \,\vert\, \gamma \in \mathcal{F} \}$ generates a right-angled Artin group for all $n \geq N^2 + N + 3$.
This extends a previous result of Koberda, who proved the existence of a bound possibly depending on the  underlying hyperbolic structure of the surface.
In the course of the proof, we obtain a universal bound depending only on the topological type of the surface in certain cases, which partially answers a question due to Koberda.
\end{abstract}

\maketitle

\section{Introduction} \label{sec:introduction}

\subsection{Background and Main Theorem}

Throughout this paper, we fix an orientable surface $\Sigma$ of genus $n_g$ with $n_p$ punctures.
An \emph{essential simple closed curve} is a circle embedded in $\Sigma$, which is not homotopic to a point and a puncture. 
We will always assume that $3 n_g + n_p - 3 > 0$, so that $\Sigma$ admits a hyperbolic structure and contains at least one essential simple closed curve.
The \emph{(geometric) intersection number} between two essential simple closed curves $\alpha$ and $\beta$ is the infimum of $\lvert \alpha' \cap \beta' \rvert$ among all simple closed curves $\alpha'$ and $\beta'$ homotopic to $\alpha$ and $\beta$, respectively.

If $\operatorname{Homeo}^+(\Sigma)$ is the group of orientation-preserving self-homeomorphisms fixing punctures pointwise and $\operatorname{Homeo}_0(\Sigma)$ is the group of self-homeomorphisms homotopic to the identity, then the \emph{mapping class group} of $\Sigma$ is defined as 
$$\operatorname{Mod}(\Sigma) = \operatorname{Homeo}^+(\Sigma)/ \operatorname{Homeo}_0(\Sigma).$$
For an essential simple closed curve $\gamma$, a (right-handed) \emph{Dehn twist} along $\gamma$ is a self-homeomorphism which is constructed by twisting a regular neighborhood of $\gamma$ once.
A more precise definition, especially the meaning of the term \emph{right-handed}, is given in Remark \ref{rem:ann_twi}; see also Farb--Margalit's book \cite{fm12}.
We write $T_\gamma$ for the mapping class of a (right-handed) Dehn twist along $\gamma$.

On the other hand, if $\Gamma = (V(\Gamma), E(\Gamma))$ is a simplicial graph, the \emph{right-angled Artin group} of $\Gamma$ is the group with the presentation $$A(\Gamma) = \langle V(\Gamma) \,\vert\, [a, b] = 1\; \text{if}\; \{ a, b \} \in E(\Gamma) \rangle.$$

Let $\operatorname{lcm}\{a, b\}$ be the least common multiple of integers $a$ and $b$.
We set $\operatorname{lcm}\{a, 0\} = 0$ for all integer $a$.
If $\mathcal{F}$ is a set of essential simple closed curves, we write $\Delta(\mathcal{F}) := \max\{\operatorname{lcm}\{i(\alpha, \beta), i(\beta, \gamma)\} \,\vert\, \alpha, \beta, \gamma \in \mathcal{F}\}$.
Note that $\Delta(\mathcal{F}) \geq \max_{\alpha, \beta \in \mathcal{F}}i(\alpha, \beta)$.
A \emph{multicurve} is a set of pairwise disjoint essential simple closed curves.
Then our main theorem is the following.

\begin{theorem}[Main theorem] \label{thm:main}
Let $\mathcal{F}$ be a nonempty finite set of essential simple closed curves on $\Sigma$ which are not homotopic to each other.
Let $\Gamma$ be the simplicial graph whose vertex set is $\mathcal{F}$ satisfying that two vertices are joined by an edge if and only if their intersection number is zero.
If $\Delta(\mathcal{F}) \geq 2$ and $m$ is an integer satisfying that
$$\lvert m \rvert \geq \max_{\substack{\alpha, \beta, \gamma \in \mathcal{F}, \\ i(\beta, \gamma) > 0}}\frac{\Delta(\mathcal{F}) + 2 \cdot i(\alpha, \gamma) + 1}{i(\beta, \gamma)} + 2,$$
then the surjective homomorphism $\varphi: A(\Gamma) \to \langle \{ T_\gamma^m \,\vert\, \gamma \in \mathcal{F} \} \rangle$ sending each alphabet $\gamma$ to $T_{\gamma}^m$ is an isomorphism.
Meanwhile, if $\Delta(\mathcal{F}) = 1$ and $\lvert m \rvert \geq 7$, then $\varphi$ is an isomorphism.
\end{theorem}

\begin{remark} \label{rem:surj}
In Theorem \ref{thm:main}, because two Dehn twists $T_\alpha^m$ and $T_\beta^m$ commute with each other if and only if $i(\alpha, \beta) = 0$, there is a surjective homomorphism $\varphi: A(\Gamma) \to \langle \{ T_\gamma^m \,\vert\, \gamma \in \mathcal{F} \} \rangle$ such that the following diagram is commutative:
\begin{center}
\begin{tikzcd}
& F(\mathcal{F}) \arrow[rd, twoheadrightarrow] \arrow[ld, twoheadrightarrow] & \\
A(\Gamma) \arrow[rr, dashrightarrow, two heads, "\varphi"] & & \langle \{ T_\gamma^m \,\vert\, \gamma \in \mathcal{F} \} \rangle
\end{tikzcd}
\end{center}
where $F(\mathcal{F})$ is the free group generated by $\mathcal{F}$.
\end{remark}

\begin{corollary} \label{cor:unif_dehn}
Let $\mathcal{F}$ be a finite set of essential simple closed curves on $\Sigma$ which are not homotopic to each other.
Assume that there is an integer $N \geq 2$ such that for every $\alpha, \beta \in \mathcal{F}$, either $i(\alpha, \beta) = 0$ or $i(\alpha, \beta) = N$.
Then the group generated by $\{ T_\gamma^6 \,\vert\, \gamma \in \mathcal{F} \}$ is isomorphic to a right-angled Artin group.
\end{corollary}

\begin{proof}
Assume that $\mathcal{F}$ is not a multicurve.
By the assumption, $\Delta(\mathcal{F}) = N$ and
$\max_{\alpha, \beta, \gamma \in \mathcal{F}} \frac{\Delta(\mathcal{F}) + 2 \cdot i(\alpha, \gamma) + 1 }{i(\beta, \gamma)} + 2 \leq 5 + \frac{1}{N}.$
By Theorem \ref{thm:main}, the group generated by $\{ T_\gamma^6 \,\vert\, \gamma \in \mathcal{F} \}$ is isomorphic to a right-angled Artin group.
\end{proof}

Corollary \ref{cor:unif_dehn} is a partial answer to Koberda's question \cite[Question 1.16]{k12b}.

\begin{corollary}
Let $\mathcal{F}$ be a finite set of essential simple closed curves on $\Sigma$ which are not homotopic to each other.
If there is an integer $N \geq 2$ such that $i(\alpha, \beta) \leq N$ for all $\alpha, \beta \in \mathcal{F}$, then the group generated by $\{ T_\gamma^m \,\vert\, \gamma \in \mathcal{F} \}$ is isomorphic to a right-angled Artin group for all integers $m \geq N^2 + N + 3$.
\end{corollary}

\begin{proof}
Assume that $\mathcal{F}$ is not a multicurve.
Since $\Delta(\mathcal{F}) \leq N(N-1)$, we have $\max_{\alpha, \beta, \gamma \in \mathcal{F}} \frac{\Delta(\mathcal{F}) + 2 \cdot i(\alpha, \gamma) + 1}{i(\beta, \gamma)} + 2 \leq \frac{N(N-1) + 2N + 1}{1} +2 = N^2 + N +3$.
Therefore, by Theorem \ref{thm:main}, the group generated by $\{ T_\gamma^m \,\vert\, \gamma \in \mathcal{F} \}$ is isomorphic to a right-angled Artin group whenever $m \geq N^2 + N + 3$.
\end{proof}

For two Dehn twists, there is a complete classification for the isomorphism type of $\langle T_\alpha^m, T_\beta^n \rangle$. (\emph{cf}. \cite[Table in Section 3.5.2]{fm12})
One can also verify from the table that the second powers of two Dehn twists generate a free abelian or free group of rank $2$.
Clay--Leininger--Mangahas \cite{clm12} showed that every right-angled Artin group can be quasi-isometrically embedded into some mapping class group of a surface of finite type.
Kim--Koberda \cite{kk16} proved that every embedding $A(\Gamma) \hookrightarrow \operatorname{Mod}(\Sigma)$ implies a graph embedding $\Gamma \hookrightarrow \mathcal{C}(\Sigma)$ into the curve graph of a surface $\Sigma$ if $\Sigma$ has low complexity.

Koberda \cite{k12b} proved that sufficiently large powers of pseudo-Anosov mapping classes on subsurfaces and Dehn twists generate a right-angled Artin group.
Kuno \cite{ku17} showed that handlebody groups contains right-angled Artin groups as subgroups.
Funar \cite{f14} proved that, if the set of simple closed curves is \emph{sparse}, then the second powers of their Dehn twists generate a right-angled Artin group.
Runnels \cite{r19} found other bounds of exponents for powers of Dehn twists in a different setting.
Crisp--Paris \cite{cp01} solved a similar problem for Artin groups.

\subsection{Hyperbolic Structure}

Let us now fix a hyperbolic structure of $\Sigma$ given by a covering map 
\[\xi : \mathbb{H}^2 \to \Sigma.\]
Unless stated otherwise, all essential simple closed curves we consider are assumed to be geodesics. In particular, if we have two distinct essential simple closed geodesics $\alpha$ and $\beta$, we have 
\[ \lvert \alpha \cap \beta \rvert = i(\alpha, \beta).\]

\subsection{The Proof of Theorem \ref{thm:main}}

Two geodesics $\tilde\alpha$ and $\tilde\beta$ on $\mathbb{H}^2$ are said to \emph{cross} each other, written by $\tilde\alpha \pitchfork \tilde\beta$, if they intersect each other transversally.
The following is a technical definition to prove Theorem \ref{thm:main}.

\begin{definition}[First definition of $\operatorname{PP}_n$] \label{def:pp_simple}
Let $\mathcal{A}$ be a multicurve on $\Sigma$.
Let $\alpha$ be a simple closed geodesic in $\mathcal{A}$.
Let $\beta$ be a simple closed geodesic crossing $\alpha$.
For an integer $n > 2$, a simple closed geodesic $\gamma$ is said to be contained in $\operatorname{PP}_n( \mathcal{A}, \alpha, \beta )$ if for some lift $\tilde\gamma$ of $\gamma$, for some lift $\tilde\alpha$ of $\alpha$ and for some (at least) $n$ lifts $\tilde\beta_1, \dots, \tilde\beta_n$ of $\beta$, the following hold.
(See Figure \ref{fig:PP_n}.)
\begin{figure}
\centering
\begingroup%
  \makeatletter%
  \providecommand\color[2][]{%
    \errmessage{(Inkscape) Color is used for the text in Inkscape, but the package 'color.sty' is not loaded}%
    \renewcommand\color[2][]{}%
  }%
  \providecommand\transparent[1]{%
    \errmessage{(Inkscape) Transparency is used (non-zero) for the text in Inkscape, but the package 'transparent.sty' is not loaded}%
    \renewcommand\transparent[1]{}%
  }%
  \providecommand\rotatebox[2]{#2}%
  \newcommand*\fsize{\dimexpr\f@size pt\relax}%
  \newcommand*\lineheight[1]{\fontsize{\fsize}{#1\fsize}\selectfont}%
  \ifx\svgwidth\undefined%
    \setlength{\unitlength}{232.63045904bp}%
    \ifx\svgscale\undefined%
      \relax%
    \else%
      \setlength{\unitlength}{\unitlength * \real{\svgscale}}%
    \fi%
  \else%
    \setlength{\unitlength}{\svgwidth}%
  \fi%
  \global\let\svgwidth\undefined%
  \global\let\svgscale\undefined%
  \makeatother%
  \begin{picture}(1,0.51326879)%
    \lineheight{1}%
    \setlength\tabcolsep{0pt}%
    \put(0,0){\includegraphics[width=\unitlength,page=1]{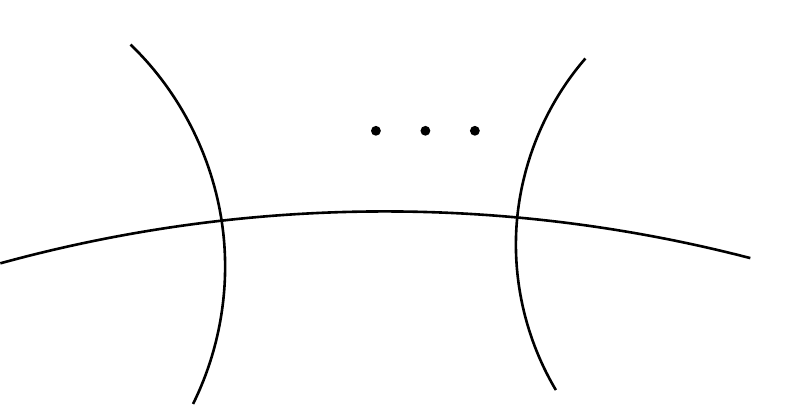}}%
    \put(0.94781562,0.17513369){\color[rgb]{0,0,0}\makebox(0,0)[lt]{\lineheight{1.25}\smash{\begin{tabular}[t]{l}$\tilde\alpha$\end{tabular}}}}%
    \put(0.94462574,0.03431083){\color[rgb]{0,0,0}\makebox(0,0)[lt]{\lineheight{1.25}\smash{\begin{tabular}[t]{l}$\tilde\gamma$\end{tabular}}}}%
    \put(0.14086995,0.4664934){\color[rgb]{0,0,0}\makebox(0,0)[lt]{\lineheight{1.25}\smash{\begin{tabular}[t]{l}$\tilde\beta_1$\end{tabular}}}}%
    \put(0.70291996,0.46649278){\color[rgb]{0,0,0}\makebox(0,0)[lt]{\lineheight{1.25}\smash{\begin{tabular}[t]{l}$\tilde\beta_n$\end{tabular}}}}%
    \put(0.31066728,0.46649349){\color[rgb]{0,0,0}\makebox(0,0)[lt]{\lineheight{1.25}\smash{\begin{tabular}[t]{l}$\tilde\beta_2$\end{tabular}}}}%
    \put(0,0){\includegraphics[width=\unitlength,page=2]{pp_n.pdf}}%
  \end{picture}%
\endgroup%

\caption{The condition for $\gamma \in \operatorname{PP}_n(\mathcal{A}, \alpha, \beta)$}
\label{fig:PP_n}
\end{figure}
\begin{enumerate}[label=(\roman*)]
\item For each $i \in \{1, \dots, n\}$, the lifts $\tilde\alpha$, $\tilde\beta_i$ and $\tilde\gamma$ cross each other.
\item If a lift $\tilde\alpha'$ of $\mathcal{A}$ crosses both $\tilde\gamma$ and some $\tilde\beta_i$, then $\tilde\alpha' = \tilde\alpha$.
\end{enumerate}
\end{definition}

The general definition of $\operatorname{PP}_n$ can be found in Definition \ref{def:ppp}.
Roughly speaking, the indicator $n$ of $\operatorname{PP}_n(\mathcal{A}, \alpha, \beta)$ is a ``combinatorial tangent'' of a simple closed geodesic on the grid of $\alpha$ (vertical) and $\beta$ (horizontal).
That is, if a simple closed geodesic $\gamma$ belongs to $\operatorname{PP}_n(\mathcal{A}, \alpha, \beta)$ with large $n$, it means that $\gamma$ follows $\alpha$ for a long time with respect to $\mathcal{A}$ and $\beta$.
It is a key notion for the proof of Theorem \ref{thm:main} to overcome difficulties.
(Compare with \cite[Section 5.7]{k12b}.)

\begin{theorem} \label{thm:pp_main}
Let $\mathcal{A}$ and $\mathcal{B}$ be multicurves on $\Sigma$ satisfying that $i(\alpha, \beta) > 0$ for some $\alpha \in \mathcal{A}$ and $\beta \in \mathcal{B}$.
Then the following hold.

\begin{enumerate}

\item \label{enum:pmn1}
If a simple closed geodesic $\gamma$ belongs to $\operatorname{PP}_n(\mathcal{A}, \alpha, \beta)$, then $$n \leq \operatorname{lcm}\{i(\alpha, \beta), i(\beta, \gamma)\}.$$

\item \label{enum:pmn2}
Let $\gamma$ be a simple closed geodesic crossing $\beta$.
If $n$ is an integer more than two and $m$ is another integer satisfying that
$$\lvert m \rvert \geq \frac{n + 2 \cdot i(\alpha, \gamma)}{i(\beta, \gamma)} + 2,$$
then $$T_\beta^m \left( \{ \alpha \} \cup \operatorname{PP}_n(\mathcal{A}, \alpha, \beta) \right) \subseteq \operatorname{PP}_n(\mathcal{B}, \beta, \gamma).$$

\item \label{enum:pmn3}
For every $\gamma \in \mathcal{A} - \{ \alpha \}$ and $m \in \mathbb{Z}$, we have $$T_{\gamma}^m(\operatorname{PP}_n(\mathcal{A}, \alpha, \beta)) \subseteq \operatorname{PP}_n(\mathcal{A}, \alpha, \beta).$$
\end{enumerate}
\end{theorem}

The proof of Theorem \ref{thm:pp_main} is on the page \pageref{pf:pp_main}.
This theorem gives a setting for ping-pong technique.
The original ping-pong lemma can be found in Koberda \cite{k12a, k12b}.

\begin{definition}
A reduced word $W = \gamma_n^{k_n} \dots \gamma_1^{k_1}$ of $A(\Gamma)$ is said to be \emph{central} if $\gamma_i \neq \gamma_j$ and $[\gamma_i, \gamma_j] = 1$ for all distinct $i, j \in \{1, \dots, n\}$.
And we say a reduced word $W$ of $A(\Gamma)$ is of $\emph{central form}$ if there exist central words $W_1$, \dots, $W_n$ such that $W = W_n \dots W_1$ and the last alphabet of $W_i$ does not commute with the last alphabet of $W_{i+1}$ for every $i \in \{1, \dots, n - 1\}$.
\end{definition}

\begin{proof}[Proof of Theorem \ref{thm:main}]
If $\Gamma$ is the join of some nonempty subgraphs $\Gamma_1$ and $\Gamma_2$, then $A(\Gamma) \cong A(\Gamma_1) \times A(\Gamma_2)$.
In this case, we can divide the proof of Theorem \ref{thm:main} into the cases for $\Gamma_1$ and $\Gamma_2$.
So we reduce this proof to the case that $\Gamma$ is not the join of two nonempty subgraphs.
In other words, we assume that each simple closed geodesic in $\mathcal{F}$ crosses another simple closed geodesic of $\mathcal{F}$.

Note that every nonidentity element of $A(\Gamma)$ admits a word of central form.
Choose $W = W_n \dots W_1 \in A(\Gamma)$ be a nonidentity reduced word of central form.
We will show that $\varphi(W)$ acts nontrivially on the set of simple closed geodesics.
For each $i = 1, \dots, n$, let $\mathrm{supp}(W_i)$ be the subset of $\mathcal{F}$ containing the alphabets of $W_i$.
Then $\mathrm{supp}(W_i)$ is a multicurve.
For each $i$, let $\gamma_i$ be the last alphabet of $W_i$.
And let $\gamma_{n+1}$ be an arbitrary simple closed geodesic in $\mathcal{F}$ crossing $\gamma_n$.
Set $\alpha := \gamma_2$.

Write $W_1 = \gamma_{1,j_1}^{k_{1,j_1}} \gamma_{1, j_1-1}^{k_{1, j_1-1}} \dots \gamma_{1,1}^{k_{1,1}} \gamma_1^{k_1}$.
Let $$N := \begin{cases} \Delta(\mathcal{F}) + 1 & \text{if } \Delta(\mathcal{F}) \geq 2, \text{ or} \\ 3 & \text{if } \Delta(\mathcal{F}) = 1. \end{cases}$$
Then we have
$$\varphi(\gamma_1^{k_1}) \alpha = T_{\gamma_1}^{mk_1} (\alpha) \in \operatorname{PP}_N(\mathrm{supp}(W_1), {\gamma_1}, \gamma_2)$$
by Theorem \ref{thm:pp_main}\eqref{enum:pmn2}.
Applying Theorem \ref{thm:pp_main}\eqref{enum:pmn3} repeatedly, $\varphi(\gamma_{1,j_1}^{k_{1,j_1}} \dots \gamma_{1,1}^{k_{1,1}})$ sends $\varphi(\gamma_1^{k_1})\alpha$ into $\operatorname{PP}_N(\mathrm{supp}(W_1), {\gamma_1}, \gamma_2 )$. Consequently, $\operatorname{PP}_N(\mathrm{supp}(W_1), {\gamma_1}, \gamma_2)$ contains $\varphi(W_1)\alpha$.

Suppose that the simple closed geodesic $\varphi(W_{l-1} W_{l-2} \dots W_1)\alpha$, denoted by $\alpha_{l-1}$, is contained in $\operatorname{PP}_N(\mathrm{supp}(W_{l-1}), {\gamma_{l-1}}, \gamma_{l})$ for some $l \in \{ 2, \dots, n \}$.
If we write $W_{l} = \gamma_{l,j_l}^{k_{l,j_l}} \dots \gamma_{l,1}^{k_{l,1}} \gamma_l^{k_l}$, then $$\varphi(\gamma_l^{k_l})\alpha_{l-1} \in \operatorname{PP}_N(\mathrm{supp}(W_l), {\gamma_l}, \gamma_{l+1})$$ by Theorem \ref{thm:pp_main}\eqref{enum:pmn2}.
And by Theorem \ref{thm:pp_main}\eqref{enum:pmn3}, we have
$$ \varphi(W_l)\alpha_{l-1} = \varphi(\gamma_{l,j_l}^{k_{l,j_l}} \dots \gamma_{l, 1}^{k_{l, 1}}) \varphi(\gamma_l^{k_l}) \alpha_{l - 1} \in \operatorname{PP}_N(\mathrm{supp}(W_l), {\gamma_l}, \gamma_{l+1}). $$
In conclusion, it holds that $\varphi(W) \alpha \in \operatorname{PP}_N(\mathrm{supp}(W_n), {\gamma_n}, \gamma_{n+1})$.

Since Theorem \ref{thm:pp_main}\eqref{enum:pmn1} implies that $\alpha$ does not belong to $\operatorname{PP}_N(\mathrm{supp}(W_n), {\gamma_n}, \gamma_{n+1})$, the action of $W$ on the set of simple closed geodesics is nontrivial.
Because $W$ is an arbitrary nonidentity reduced word of $A(\Gamma)$, the action of $A(\Gamma)$ on the set of simple closed geodesics is faithful.
Therefore, the surjective homomorphism $\varphi$ is injective, i.e., it is an isomorphism.
\end{proof}

\subsection{Guide to Readers}

Although we closely follow the original approach of Koberda \cite{k12b}, the material of this paper is essentially self-contained.
As we have just shown that Theorem \ref{thm:pp_main} implies Theorem \ref{thm:main}, we will from now only focus on the proof of the former.
Dual trees and the set $\operatorname{PP}_n$ are our main tools.
Section \ref{sec:act_G} is an introduction to the actions of fundamental groups of surfaces on dual trees.
Proposition \ref{prop:num_lift} is the main result of this section.
In Section \ref{sec:normalizer}, we study the action of lifts of Dehn twists on dual trees.
Proposition \ref{prop:dehn_cal} gives how we compute Dehn twists by fundamental groups of surfaces.
In Section \ref{sec:fpc}, we investigate the set $\operatorname{PP}_n$ and prove Theorem \ref{thm:pp_main}.
Proposition \ref{prop:half_ele} is the technical essence of our paper.

\subsection{Acknowledgement}
The author is supported by Samsung Science and Technology Foundation (SSTF-BA1301-51). I would like to thank my advisor Sang-hyun Kim for many suggestions regarding this work.
I am grateful to Thomas Koberda for writing the main inspiration~\cite{k12b} of this paper, and also for many encouraging comments. I appreciate Runnel for helpful discussions. 

\section{The Actions of Fundamental Groups of Surfaces on Dual Trees} \label{sec:act_G}

For disjoint subspaces $A$ and $B$ of $\mathbb{H}^2$, we say that a geodesic $L$ \emph{separates} $A$ from $B$ if $A$ and $B$ lie in different connected components of $\mathbb{H}^2 \setminus L$.

\begin{remark} \label{rem:sep}
For pairwise disjoint geodesics $L_1, L_2, L_3$ on $\mathbb{H}^2$, we will use the well-known facts.
\begin{enumerate}
\item Every geodesic crossing both $L_1$ and $L_3$ also crosses $L_2$ if and only if $L_2$ separates $L_1$ from $L_3$.
\item If some geodesic crosses $L_1$, $L_2$ and $L_3$, then some $L_i$ separates the others.
\end{enumerate}
\end{remark}

\begin{definition} \label{def:dual}
Let $\xi: \mathbb{H}^2 \to \Sigma$ be a covering map.
For a simple closed geodesic $\gamma$ on $\Sigma$, the \emph{dual tree} of $\gamma$ is a triple $(\mathcal{Y}_\gamma, d_\gamma, \sigma_\gamma)$ satisfying the following.
(We write $\mathcal{Y}_\gamma$ simply for the triple $(\mathcal{Y}_\gamma, d_\gamma, \sigma_\gamma)$.)
\begin{itemize}
\item $\mathcal{Y}_\gamma$ is a topological tree embedded into $\mathbb{H}^2$ satisfying that every connected component of $\mathbb{H}^2 \setminus \xi^{-1}(\gamma)$ contains exactly one vertex of $\mathcal{Y}_\gamma$ and for every lift $\tilde\gamma$ of $\gamma$, there is a unique edge $e$ on $\mathcal{Y}_\gamma$ such that $\lvert \mathcal{Y}_\gamma \cap \tilde\gamma \rvert = \lvert e \cap \tilde\gamma \rvert = \lvert e \cap \xi^{-1}(\gamma) \rvert = 1$.
(See Figure \ref{fig:emb_cub}.)
\item $d_\gamma$ is a metric on $\mathcal{Y}_\gamma$ satisfying that for each edge $e$ of $\mathcal{Y}_\gamma$, the length of $e$ is $1$ and the intersection between $e$ and a lift of $\gamma$ is the midpoint of $e$.
\item $\sigma_\gamma$ is an isometric action of $\pi_1(\Sigma)$ on $(\mathcal{Y}_\gamma, d_\gamma)$ satisfying that $\sigma_\gamma(g)(\tilde\gamma \cap \mathcal{Y}_\gamma) = (g\tilde\gamma) \cap \mathcal{Y}_\gamma$ for all $g \in \pi_1(\Sigma)$ and a lift $\tilde\gamma$ of $\gamma$.
\end{itemize}
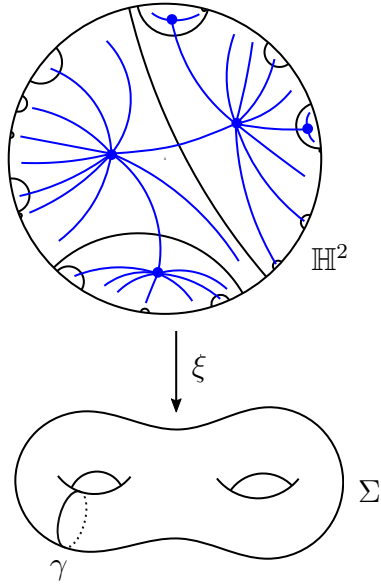
\begin{figure}
\centering
\begingroup%
  \makeatletter%
  \providecommand\color[2][]{%
    \errmessage{(Inkscape) Color is used for the text in Inkscape, but the package 'color.sty' is not loaded}%
    \renewcommand\color[2][]{}%
  }%
  \providecommand\transparent[1]{%
    \errmessage{(Inkscape) Transparency is used (non-zero) for the text in Inkscape, but the package 'transparent.sty' is not loaded}%
    \renewcommand\transparent[1]{}%
  }%
  \providecommand\rotatebox[2]{#2}%
  \newcommand*\fsize{\dimexpr\f@size pt\relax}%
  \newcommand*\lineheight[1]{\fontsize{\fsize}{#1\fsize}\selectfont}%
  \ifx\svgwidth\undefined%
    \setlength{\unitlength}{144.94153481bp}%
    \ifx\svgscale\undefined%
      \relax%
    \else%
      \setlength{\unitlength}{\unitlength * \real{\svgscale}}%
    \fi%
  \else%
    \setlength{\unitlength}{\svgwidth}%
  \fi%
  \global\let\svgwidth\undefined%
  \global\let\svgscale\undefined%
  \makeatother%
  \begin{picture}(1,1.49226391)%
    \lineheight{1}%
    \setlength\tabcolsep{0pt}%
    \put(0,0){\includegraphics[width=\unitlength,page=1]{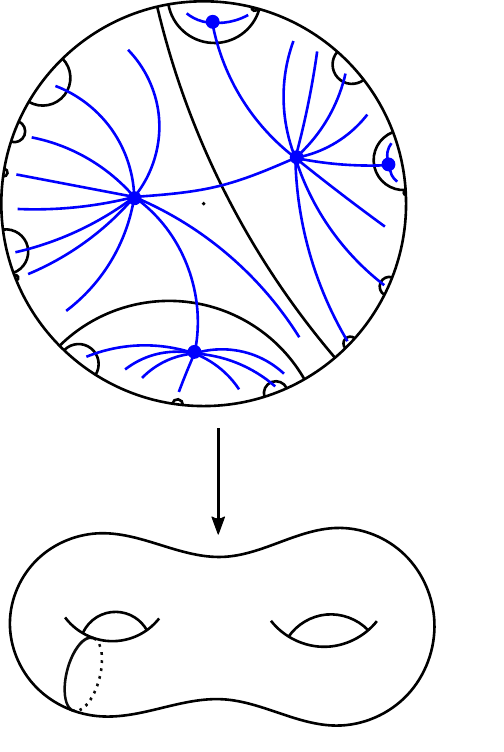}}%
    \put(0.47348161,0.52361624){\color[rgb]{0,0,0}\makebox(0,0)[lt]{\lineheight{1.25}\smash{\begin{tabular}[t]{l}$\xi$\end{tabular}}}}%
    \put(0.10813973,0.01139999){\color[rgb]{0,0,0}\makebox(0,0)[lt]{\lineheight{1.25}\smash{\begin{tabular}[t]{l}$\gamma$\end{tabular}}}}%
    \put(0.90154916,0.19800615){\color[rgb]{0,0,0}\makebox(0,0)[lt]{\lineheight{1.25}\smash{\begin{tabular}[t]{l}$\Sigma$\end{tabular}}}}%
    \put(0.78161017,0.79989856){\color[rgb]{0,0,0}\makebox(0,0)[lt]{\lineheight{1.25}\smash{\begin{tabular}[t]{l}$\mathbb{H}^2$\end{tabular}}}}%
  \end{picture}%
\endgroup%

\caption{The dual tree of a simple closed geodesic $\gamma$ \label{fig:emb_cub}}
\end{figure}
\end{definition}

Note that if $p \in \mathcal{Y}_\gamma$ is not the midpoint of an edge, then for every $g \in \pi_1(\Sigma)$, two points $\sigma_\gamma(g)p$ (the isometric action on $(\mathcal{Y}_\gamma, d_\gamma)$) and $gp$ (the isometric action on $(\mathbb{H}^2, d_{\mathbb{H}^2})$) can be different but they are always contained in the same connected component of $\mathbb{H}^2 \setminus \xi^{-1}(\gamma)$.
Note that an element of $\pi_1(\Sigma)$ preserving a lift of $\gamma$ fixes an edge of $\mathcal{Y}_\gamma$.

For a simple closed geodesic $\gamma$ and its lift $\tilde\gamma$, we write $\mathcal{N}_r(\gamma)$ and $\mathcal{N}_r(\tilde\gamma)$ for the open $r$-neighborhoods of $\gamma$ and $\tilde\gamma$, respectively.
And we write $\overline{\mathcal{N}_r(\gamma)}$ and $\overline{\mathcal{N}_r(\tilde\gamma)}$ for the closure of $\mathcal{N}_r(\gamma)$ and $\mathcal{N}_r(\tilde\gamma)$, respectively.
The collar lemma (\emph{cf.} \cite{b78}) means that for every simple closed geodesic $\gamma$ on $\Sigma$, there is a positive number $R$, depending only on the length of $\gamma$ such that $\overline{\mathcal{N}_r(\gamma)}$ is homeomorphic to an annulus for every $0 < r < R$.

\begin{proposition} \label{prop:hyp_tree}
Let $\gamma$ be a simple closed geodesic on $\Sigma$.
Let $R$ be a positive number such that $\overline{\mathcal{N}_r(\gamma)}$ is homeomorphic to an annulus for every $0 < r < R$.
Then for every $0 < r < R$, there is a $\pi_1(\Sigma)$-equivariant surjective continuous map $\Phi_{\gamma, r}: \mathbb{H}^2 \to \mathcal{Y}_\gamma$ such that the following hold.
\begin{enumerate}
\item \label{enum:sur1} For each lift $\tilde\gamma$ of $\gamma$, if $e(\tilde\gamma)$ is the edge intersecting $\tilde\gamma$ and $u$ is a point in the interior of $e(\tilde\gamma)$, then $\Phi_{\gamma, r}^{-1}(u)$ is a one-dimensional subspace of $\mathcal{N}_r(\tilde\gamma)$.
\item \label{enum:sur2} If $\alpha$ is a simple closed geodesic on $\Sigma$, then there is $0 < R' \leq R$ such that for every $0 < r < R'$ and a lift $\tilde\alpha$ of $\alpha$, the image $\Phi_{\gamma, r}(\tilde\alpha)$ is given as follows:
\begin{enumerate}
\item the midpoint of an edge if $\alpha = \gamma$;
\item a vertex if $\alpha \cap \gamma = \emptyset$;
\item a geodesic with respect to $d_\gamma$ if $\alpha \pitchfork \gamma$.
\end{enumerate}
\end{enumerate}
\end{proposition}

\begin{proof}
\noindent\eqref{enum:sur1}
Fix a positive number $r < R$.
Let $\rho: S^1 \times [0, 1] \to \Sigma$ be a topological embedding satisfying that $\operatorname{im}\rho = \overline{\mathcal{N}_r(\gamma)}$ and $\rho(S^1 \times \{1/2\}) = \gamma$.
Let $\xi: \mathbb{H}^2 \to \Sigma$ be a covering map.
For each lift $\tilde\gamma$ of $\gamma$, let $\rho_{\tilde\gamma}$ be a lift of $\rho$ whose image contains $\tilde\gamma$ as a subset.
And let $\psi_{\tilde\gamma}: [0, 1] \to e(\tilde\gamma)$ be an isometry satisfying that $\psi_{\tilde\gamma}(0)$ and $\rho_{\tilde\gamma}(s, 0)$ are in the same connected component of $\mathbb{H}^2 \setminus \xi^{-1}(\gamma)$ for all $s \in S^1$.

Then for every $g \in \pi_1(\Sigma)$ and $t \in [0, 1]$, the point $\sigma_\gamma(g) \psi_{\tilde\gamma}(t)$ is on the edge $e(g\tilde\gamma)$.
So $\sigma_\gamma(g) \psi_{\tilde\gamma}$ is an isometry from $[0, 1]$ to $e(g\tilde\gamma)$.
Because $\sigma_\gamma(g) \psi_{\tilde\gamma}(0)$ and $g\psi_{\tilde\gamma}(0)$ are in the same connected component of $\mathbb{H}^2 \setminus \xi^{-1}(\gamma)$, we have $\sigma_\gamma(g) \psi_{\tilde\gamma} = \psi_{g\tilde\gamma}$.

With the map $\mathrm{pr}: (s, t) \mapsto t$, we define a map $\Phi_{\gamma, r}: \mathbb{H}^2 \to \mathcal{Y}_\gamma$ by
$$\Phi_{\gamma, r}(p) = \begin{cases} ( \psi_{\tilde\gamma} \circ \mathrm{pr} \circ \rho_{\tilde\gamma}^{-1} ) ( p) & \text{if } p \in \operatorname{im} \rho_{\tilde\gamma} \text{ for some lift } \tilde\gamma, \\ \text{the vertex lying in the same component}, & \text{otherwise.} \end{cases}$$
Then $\Phi_{\gamma, r}$ is well-defined, continuous and surjective, and satisfies the condition \eqref{enum:sur1}.

To show that $\Phi_{\gamma, r}$ is $\pi_1(\Sigma)$-equivariant, choose a point $q$ on $\mathbb{H}^2$ and $g \in \pi_1(\Sigma)$.
If $q$ is a point in $\overline{\mathcal{N}_r(\tilde\gamma)} = \operatorname{im}\rho_{\tilde\gamma}$ for some lift $\tilde\gamma$, then $gq \in \overline{\mathcal{N}_r(g\tilde\gamma)} = \operatorname{im}\rho_{g\tilde\gamma}$.
Because $\rho_{\tilde\gamma}$ and $\rho_{g\tilde\gamma}$ are lifts of $\rho$, we have $(\mathrm{pr} \circ \rho_{\tilde\gamma}^{-1})(q) = (\mathrm{pr} \circ \rho_{g\tilde\gamma}^{-1})(gq)$.
Then $\sigma_\gamma(g) \Phi_{\gamma, r}(q) = \sigma_\gamma(g) (\psi_{\tilde\gamma}(\mathrm{pr}\circ \rho_{\tilde\gamma}^{-1})(q)) = \psi_{g\tilde\gamma}(\mathrm{pr} \circ \rho_{g\tilde\gamma}^{-1}(gq)) = \Phi_{\gamma, r}(gq)$.

If $\Phi_{\gamma, r}(q)$ is a vertex $v$ on $\mathcal{Y}_\gamma$, then $q$ and $v$ are in the same connected component of $\mathbb{H}^2 \setminus \xi^{-1}(\gamma)$.
Since $gq$, $gv$ and $\sigma_\gamma(g) v$ are in the same connected component, we have $\sigma_\gamma(g) \Phi_{\gamma, r}(q) = \sigma_\gamma(g) v = \Phi_{\gamma, r}(gq)$.
Therefore, $\Phi_{\gamma, r}$ is $\pi_1(\Sigma)$-equivariant. \vspace{1.5mm}

\noindent\eqref{enum:sur2}
If $\alpha = \gamma$, then $\Phi_{\gamma, r}(\tilde\alpha)$ is the midpoint of an edge by definition.
If $\alpha$ is disjoint from $\gamma$, then $\alpha$ is also disjoint from $\mathcal{N}_r(\gamma)$ for all $0< r < R(\gamma)$ by the collar lemma.
So $\tilde\alpha$ is disjoint from the $r$-neighborhoods of all lifts of $\gamma$.
Therefore, $\Phi_{\gamma, r}(\tilde\alpha)$ is a vertex.

Assume that $\alpha$ crosses $\gamma$. 
Since $\alpha$ is compact, there is a positive number $R(\alpha, \gamma) > 0$ such that $\alpha \cap \overline{\mathcal{N}_r(\gamma)}$ is the disjoint union of line segments intersecting $\gamma$ for all $0 < r < R(\alpha, \gamma)$.
So $\tilde\alpha$ crosses $\tilde\gamma$ whenever a lift $\tilde\alpha$\ of $\alpha$ intersects the $r$-neighborhood of a lift $\tilde\gamma$ of $\gamma$.
Therefore, $\Phi_\gamma(\tilde\alpha)$ is the union of edges, which is a geodesic on $\mathcal{Y}_\gamma$.
\end{proof}

The \emph{translation length}  of an isometry $f$ on the dual tree $\mathcal{Y}_\gamma$ is
$$ \operatorname{tr}_\gamma f := \inf_{v \in \mathcal{Y}_\gamma} (v, fv). $$
For $g \in \pi_1(\Sigma)$, we write $\operatorname{tr}_\gamma g$ instead of $\operatorname{tr}_\gamma \sigma_\gamma(g)$.
An element of $\pi_1(\Sigma)$ is said to be \emph{primitive} if it cannot be written by a proper power of another element of $\pi_1(\Sigma)$.

\begin{lemma}
If $\gamma$ is a simple closed geodesic on $\Sigma$ and $h$ is a primitive element of $\pi_1(\Sigma)$ preserving a lift of a simple closed geodesic $\alpha$, then $\mathrm{tr}_\gamma\, h =i(\alpha, \gamma)$.
\end{lemma}

\begin{proof}
Let $\tilde\alpha$ be the lift of $\alpha$ preserved by $h$.
If $i(\alpha, \gamma) = 0$, then $\Phi_{\gamma, r}(\tilde\alpha)$ is a point on $\mathcal{Y}_\gamma$ for some $r > 0$.
Because $h$ preserves $\Phi_{\gamma, r}(\tilde\alpha)$, we have $\operatorname{tr}_\gamma h = 0 = i(\alpha, \gamma)$.

Assume that $\alpha$ crosses $\gamma$.
By Proposition \ref{prop:hyp_tree}\eqref{enum:sur2}, it is satisfied that $h$ preserves the geodesic $\Phi_{\gamma, r}(\tilde\alpha)$ for some $r > 0$.
If $p$ is a point on $\tilde\alpha$ such that $\Phi_{\gamma, r}(p)$ is a vertex, then $\operatorname{tr}_\gamma h $ is equal to $d_\gamma(\Phi_\gamma(p), \sigma_\gamma(h) \Phi_\gamma(p))$ by Bass-Serre theory.
Since the geodesic segment joining $p$ and $h p$ crosses exactly $i(\alpha, \gamma)$ lifts of $\gamma$, we have $d_\gamma(\Phi_{\gamma, r}(p), \Phi_{\gamma, r}(h p)) = i(\alpha, \gamma)$.
Therefore, $\operatorname{tr}_\gamma h = i(\alpha, \gamma)$.
\end{proof}

\begin{lemma} \label{lem:tria}
Let $\alpha$, $\beta$ and $\gamma$ be simple closed geodesics on $\Sigma$ crossing each other.
If $x \in \alpha \cap \beta$, $y \in \beta \cap \gamma$ and $z \in \gamma \cap \alpha$, then the number of distinct contractible geodesic triangles whose vertices are $x$, $y$ and $z$ is at most $1$.
\end{lemma}

\begin{proof}
\begin{figure}
\centering
\subfloat[][]{
\begingroup%
  \makeatletter%
  \providecommand\color[2][]{%
    \errmessage{(Inkscape) Color is used for the text in Inkscape, but the package 'color.sty' is not loaded}%
    \renewcommand\color[2][]{}%
  }%
  \providecommand\transparent[1]{%
    \errmessage{(Inkscape) Transparency is used (non-zero) for the text in Inkscape, but the package 'transparent.sty' is not loaded}%
    \renewcommand\transparent[1]{}%
  }%
  \providecommand\rotatebox[2]{#2}%
  \newcommand*\fsize{\dimexpr\f@size pt\relax}%
  \newcommand*\lineheight[1]{\fontsize{\fsize}{#1\fsize}\selectfont}%
  \ifx\svgwidth\undefined%
    \setlength{\unitlength}{117.29532677bp}%
    \ifx\svgscale\undefined%
      \relax%
    \else%
      \setlength{\unitlength}{\unitlength * \real{\svgscale}}%
    \fi%
  \else%
    \setlength{\unitlength}{\svgwidth}%
  \fi%
  \global\let\svgwidth\undefined%
  \global\let\svgscale\undefined%
  \makeatother%
  \begin{picture}(1,0.67761407)%
    \lineheight{1}%
    \setlength\tabcolsep{0pt}%
    \put(0,0){\includegraphics[width=\unitlength,page=1]{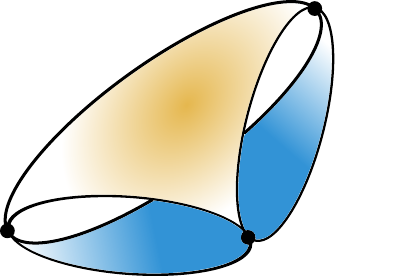}}%
    \put(0.27164298,0.31120437){\color[rgb]{0,0,0}\makebox(0,0)[lt]{\lineheight{1.25}\smash{\begin{tabular}[t]{l}$D_1$\end{tabular}}}}%
    \put(0.81876883,0.22066503){\color[rgb]{0,0,0}\makebox(0,0)[lt]{\lineheight{1.25}\smash{\begin{tabular}[t]{l}$D_2$\end{tabular}}}}%
    \put(0,0){\includegraphics[width=\unitlength,page=2]{two_tria.pdf}}%
  \end{picture}%
\endgroup%
 \label{subfig:two1}}
\subfloat[][]{
\begingroup%
  \makeatletter%
  \providecommand\color[2][]{%
    \errmessage{(Inkscape) Color is used for the text in Inkscape, but the package 'color.sty' is not loaded}%
    \renewcommand\color[2][]{}%
  }%
  \providecommand\transparent[1]{%
    \errmessage{(Inkscape) Transparency is used (non-zero) for the text in Inkscape, but the package 'transparent.sty' is not loaded}%
    \renewcommand\transparent[1]{}%
  }%
  \providecommand\rotatebox[2]{#2}%
  \newcommand*\fsize{\dimexpr\f@size pt\relax}%
  \newcommand*\lineheight[1]{\fontsize{\fsize}{#1\fsize}\selectfont}%
  \ifx\svgwidth\undefined%
    \setlength{\unitlength}{131.24674159bp}%
    \ifx\svgscale\undefined%
      \relax%
    \else%
      \setlength{\unitlength}{\unitlength * \real{\svgscale}}%
    \fi%
  \else%
    \setlength{\unitlength}{\svgwidth}%
  \fi%
  \global\let\svgwidth\undefined%
  \global\let\svgscale\undefined%
  \makeatother%
  \begin{picture}(1,0.57489846)%
    \lineheight{1}%
    \setlength\tabcolsep{0pt}%
    \put(0,0){\includegraphics[width=\unitlength,page=1]{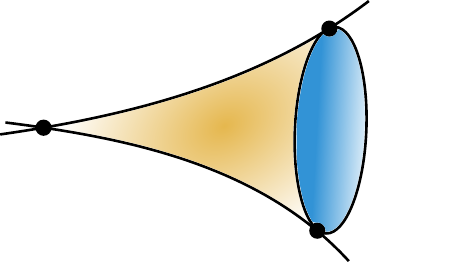}}%
    \put(0.45914123,0.26117036){\color[rgb]{0,0,0}\makebox(0,0)[lt]{\lineheight{1.25}\smash{\begin{tabular}[t]{l}$D_1$\end{tabular}}}}%
    \put(0.82225059,0.26016838){\color[rgb]{0,0,0}\makebox(0,0)[lt]{\lineheight{1.25}\smash{\begin{tabular}[t]{l}$D_2$\end{tabular}}}}%
  \end{picture}%
\endgroup%
 \label{subfig:two2}}
\subfloat[][]{
\begingroup%
  \makeatletter%
  \providecommand\color[2][]{%
    \errmessage{(Inkscape) Color is used for the text in Inkscape, but the package 'color.sty' is not loaded}%
    \renewcommand\color[2][]{}%
  }%
  \providecommand\transparent[1]{%
    \errmessage{(Inkscape) Transparency is used (non-zero) for the text in Inkscape, but the package 'transparent.sty' is not loaded}%
    \renewcommand\transparent[1]{}%
  }%
  \providecommand\rotatebox[2]{#2}%
  \newcommand*\fsize{\dimexpr\f@size pt\relax}%
  \newcommand*\lineheight[1]{\fontsize{\fsize}{#1\fsize}\selectfont}%
  \ifx\svgwidth\undefined%
    \setlength{\unitlength}{89.82268326bp}%
    \ifx\svgscale\undefined%
      \relax%
    \else%
      \setlength{\unitlength}{\unitlength * \real{\svgscale}}%
    \fi%
  \else%
    \setlength{\unitlength}{\svgwidth}%
  \fi%
  \global\let\svgwidth\undefined%
  \global\let\svgscale\undefined%
  \makeatother%
  \begin{picture}(1,0.96780763)%
    \lineheight{1}%
    \setlength\tabcolsep{0pt}%
    \put(0,0){\includegraphics[width=\unitlength,page=1]{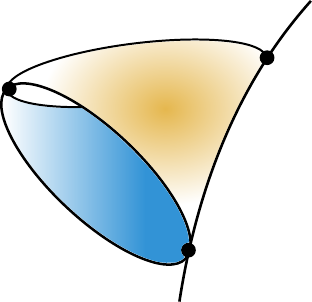}}%
    \put(0.44065018,0.63385042){\color[rgb]{0,0,0}\makebox(0,0)[lt]{\lineheight{1.25}\smash{\begin{tabular}[t]{l}$D_1$\end{tabular}}}}%
    \put(0.14868948,0.4187626){\color[rgb]{0,0,0}\makebox(0,0)[lt]{\lineheight{1.25}\smash{\begin{tabular}[t]{l}$D_2$\end{tabular}}}}%
  \end{picture}%
\endgroup%
 \label{subfig:two3}}
\caption{\label{fig:two_tria}}
\end{figure}
For contradiction, suppose that there are distinct disks $D_1$ and $D_2$ whose boundaries are geodesic triangles such that $x$, $y$ and $z$ are vertices of $\partial D_i$ for all $i = 1, 2$.
Write $\partial D_i$ for the boundary of $D_i$.

If $\partial D_1 \cap \partial D_2$ does not contain a side of $\partial D_1$ or $\partial D_2$ (\emph{cf.} Figure \ref{subfig:two1}), then $D_1 \cup D_2$ is homotopic to a pair of pants.
It implies that $\alpha$, $\beta$ and $\gamma$ are pairwise disjoint, which is a contradiction.

Assume that $\partial D_1$ and $\partial D_2$ share exactly two sides. (\emph{cf.} Figure \ref{subfig:two2})
Then $D_1 \cup D_2$ is a disk and its boundary is one of $\alpha$, $\beta$ or $\gamma$.
It induces that one of $\alpha$, $\beta$ and $\gamma$ is contractible to a point, which is a contradiction.

If $\partial D_1 \cap \partial D_2$ is the disjoint union of a vertex and a side (\emph{cf.} Figure \ref{subfig:two3}), then $D_1 \cup D_2$ is homotopic to a cylinder.
So two of $\alpha$, $\beta$ and $\gamma$ are homotopic to each other, which is a contradiction.
Therefore, $\partial D_1$ and $\partial D_2$ share all sides, which implies that $D_1 = D_2$.
\end{proof}

\begin{definition}
If $\gamma$ is a simple closed geodesic on $\Sigma$ and $L_1$ and $L_2$ are geodesics on $\mathbb{H}^2$, then we write $\Delta_\gamma(L_1, L_2)$ for the number of lifts of $\gamma$ crossing both $L_1$ and $L_2$.
\end{definition}

\begin{remark}
Let $\alpha$, $\beta$ and $\gamma$ be simple closed geodesics on $\Sigma$.
For sufficiently small $r > 0$, a lift $\tilde\alpha$ of $\alpha$ crosses a lift $\tilde\gamma$ of $\gamma$ if and only if the edge containing the midpoint $\Phi_{\gamma, r}(\tilde\gamma)$ belongs to $\Phi_{\gamma, r}(\tilde\alpha)$.
For a lift $\tilde\beta$ of $\beta$, the number $\Delta_\gamma(\tilde\alpha, \tilde\beta)$ is equal to the length of $\Phi_{\gamma, r}(\tilde\alpha) \cap \Phi_{\gamma, r}(\tilde\beta)$.
\end{remark}

\begin{proposition} \label{prop:num_lift}
Let $\alpha$, $\beta$ and $\gamma$ be simple closed geodesics on $\Sigma$.
And let $\tilde\alpha$ and $\tilde\beta$ be lifts of $\alpha$ and $\beta$, respectively.
\begin{enumerate}
\item \label{enum:dis_num}
If $\tilde\alpha$ is disjoint from $\tilde\beta$, then $\Delta_\gamma(\tilde\alpha, \tilde\beta) \leq \min\{ i(\alpha, \gamma), i(\beta, \gamma) \}$.
\item \label{enum:cro_num}
If $\tilde\alpha$ crosses $\tilde\beta$, then $\Delta_\gamma(\tilde\alpha, \tilde\beta) \leq \operatorname{lcm}\{ i(\alpha, \gamma), i(\beta, \gamma) \}$.
\end{enumerate}
\end{proposition}

\begin{proof}
\noindent\eqref{enum:dis_num}
Let $r > 0$ be sufficiently small, and let $\Phi_\gamma = \Phi_{\gamma, r}: \mathbb{H}^2 \to \mathcal{Y}_\gamma$ be a $\pi_1(\Sigma)$-equivariant surjective continuous map satisfying Proposition \ref{prop:hyp_tree}.
Suppose that either $i(\alpha, \gamma)$ or $i(\beta, \gamma)$ is zero.
Then $\Phi_\gamma(\tilde\alpha) \cap \Phi_\gamma(\tilde\beta)$ has at most one element.
So $\Delta_\gamma(\tilde\alpha, \tilde\beta) = 0$.

Assume that both $i(\alpha, \gamma)$ and $i(\beta, \gamma)$ are positive.
Let $h$ be a primitive element of $\pi_1(\Sigma)$ preserving $\tilde\alpha$.
Because $h$ is orientation-preserving, $\tilde\alpha$ does not separate $\tilde\beta$ from $h \tilde\beta$.
Since $h$ is an isometry, $h\tilde\beta$ does not separate $\tilde\alpha$ from $\tilde\beta$.
Likewise, $\tilde\beta$ does not separate $\tilde\alpha$ from $h \tilde\beta$.
By Remark \ref{rem:sep}, for every lift of $\gamma$, it is disjoint from some of $\tilde\alpha$, $\tilde\beta$ and $h \tilde\beta$.

Then $\Phi_\gamma(\tilde\alpha) \cap \Phi_\gamma(\tilde\beta) \cap \Phi_\gamma(h\tilde\beta)$ is a vertex or is empty.
So the intersection between $\Phi_\gamma(\tilde\alpha) \cap \Phi_\gamma(\tilde\beta)$ and $\Phi_\gamma(\tilde\alpha) \cap \Phi_\gamma(h\tilde\beta) = \sigma_\gamma(h)(\Phi_\gamma(\tilde\alpha) \cap \Phi_\gamma(\tilde\beta))$ does not contain an edge.
Similarly, $\Phi_\gamma(\tilde\alpha) \cap \Phi_\gamma(\tilde\beta)$ does not share an edge with $\sigma_\gamma(h^{-1})(\Phi_\gamma(\tilde\alpha) \cap \Phi_\gamma(\tilde\beta))$.
It implies that $\Phi_\gamma(\tilde\alpha) \cap \Phi_\gamma(\tilde\beta)$ is the line segment of length at most $\operatorname{tr}_\gamma h = i(\alpha, \gamma)$.

In other words, the number of lifts of $\gamma$ intersecting $\tilde\alpha$ and $\tilde\beta$ simultaneously is at most $i(\alpha, \gamma)$.
Changing the role between $\alpha$ and $\beta$, we have $\Delta_\gamma(\tilde\alpha, \tilde\beta) \leq i(\beta, \gamma)$.
Therefore, $\Delta_{\gamma}(\tilde\alpha, \tilde\beta) \leq \min \{ i(\alpha, \gamma), i(\beta, \gamma) \}$. \vspace{1.5mm}

\noindent\eqref{enum:cro_num}
Write $N := \operatorname{lcm}\{ i(\alpha, \gamma), i(\beta, \gamma) \}$, $a := N / i(\alpha, \gamma)$ and $b := N / i(\beta, \gamma)$.
For contradiction, assume that $\Delta_\gamma(\tilde\alpha, \tilde\beta) \geq N + 1$.
Let $h$ be a primitive element of $\pi_1(\Sigma)$ preserving $\tilde\alpha$.
By Proposition \ref{prop:hyp_tree}\eqref{enum:sur2} and the assumption, $\Phi_\gamma(\tilde\alpha)$ and $\Phi_\gamma(\tilde\beta)$ are geodesics on $\mathcal{Y}_\gamma$ and their intersection is a line segment of length at least $N + 1$.
So there is an edge $e$ on $\Phi_\gamma(\tilde\alpha) \cap \Phi_\gamma(\tilde\beta)$ such that $\sigma_\gamma(h^a)e \subset \Phi_\gamma(\tilde\alpha) \cap \Phi_\gamma(\tilde\beta)$.
In other words, there is a lift $\tilde\gamma$ of $\gamma$ such that both $\tilde\gamma$ and $h^a \tilde\gamma$ crosses both $\tilde\alpha$ and $\tilde\beta$.
See Figure \ref{fig:cro_tria}.
\begin{figure}
\centering
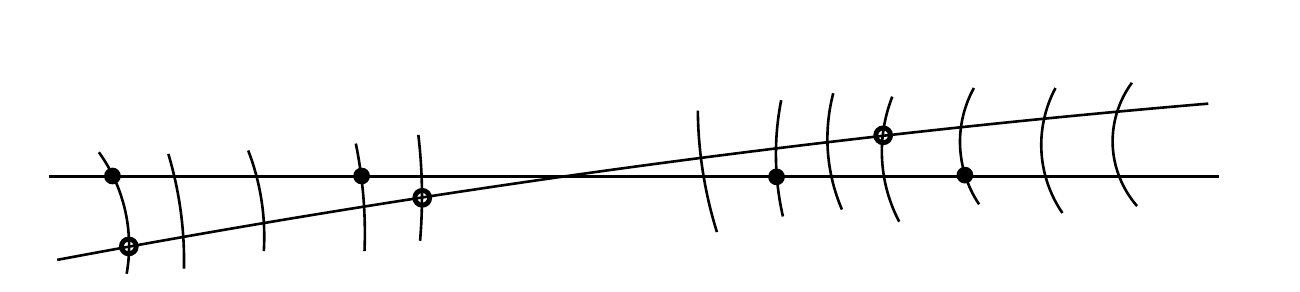
\caption{\label{fig:cro_tria}}
\end{figure}

Because $d_\gamma(\Phi_\gamma(\tilde\gamma), \Phi_\gamma(h^a\tilde\gamma))$ is a multiple of $i(\beta, \gamma)$, there is a primitive element $g$ of $\pi_1(\Sigma)$ preserving $\tilde\beta$ such that $g^b\tilde\gamma = h^a \tilde\gamma$.
Let $A$ and $B$ be the geodesic triangles which are contained in  $\tilde\alpha \cup \tilde\beta \cup \tilde\gamma$ and $\tilde\alpha \cup \tilde\beta \cup g^b\tilde\gamma$, respectively.
If $\xi: \mathbb{H}^2 \to \Sigma$ is a covering map, $\xi(A)$ and $\xi(B)$ are distinct contractible geodesic triangles on $\Sigma$ such that they share all vertices.
It is a contradiction because of Lemma \ref{lem:tria}.
Therefore, $\Delta_\gamma(\tilde\alpha, \tilde\beta) \leq N = \operatorname{lcm}\{i(\alpha, \gamma), i(\beta, \gamma)\}$.
\end{proof}

\begin{remark}[Further reading]
Morgan--Shalen \cite{ms91} defined the dual trees for measured laminations, which is more general than our definition.
Levitt--Paulin \cite{lp97} generally studied geometric actions on trees.
Skora \cite{s96} showed that nontrivial minimal tree actions of hyperbolic surface groups such that every parabolic isometry fixes a vertex are given from dual trees of measured laminations.
Rips and Bestvina--Feign \cite[Theorem 9.8]{bf95} proved that every finitely generated group acting freely on a tree is a free product of surface groups and free abelian groups.

The theory of tree action is originated from Bass--Serre's work \cite{s77, b93}.
It had been developed by Serre, Bass, Dunwoody (\emph{cf.} \cite{dd89}) and so on.
Sageev's study \cite{s95} provided methods for $\rm CAT(0)$ cube complexes including dual trees of simple closed geodesics.
Haglund--Wise \cite{hw08, hw12} developed special cube complexes.
It built a foundation for Agol \cite{a13} to prove the virtual Haken conjecture.
\end{remark}

\section{Lifts of Dehn twists} \label{sec:normalizer}

\begin{remark}
In this section, we write $T_\gamma$ as a self-homeomorphism homotopic to a Dehn twist along $\gamma$.
\end{remark}

It is well-known that the mapping class group $\mathrm{Mod}(\Sigma)$ is a subgroup of the outer automorphism group of $\pi_1(\Sigma)$.
Note that, when $\Sigma$ is closed, $\mathrm{Mod}(\Sigma)$ is an index $2$ subgroup of $\mathrm{Out}(\pi_1(\Sigma))$ by the Dehn--Niesen--Baer theorem; see Farb--Margalit \cite[Theorem 8.1]{fm12}.

\begin{remark} \label{rem:mcg_out}
Let us recall the relationship between mapping classes and outer automorphisms.
Let $F$ be a self-homeomorphism of $\Sigma$.
If $\xi: \mathbb{H}^2 \to \Sigma$ is a covering map and $\tilde{F}$ is a self-homeomorphism of $\mathbb{H}^2$ which is a lift of $F$, then $\xi(\tilde{F} g \tilde{F}^{-1} p) = F\xi( g \tilde{F}^{-1} p) = F \xi(\tilde{F}^{-1}p) = \xi(p)$ for every $g \in \pi_1(\Sigma)$ and $p \in \mathbb{H}^2$.
That is, $\tilde{F} g \tilde{F}^{-1}$ is an element of $\pi_1(\Sigma)$, and the conjugate action of $\tilde{F}$ on $\pi_1(\Sigma)$ is an automorphism of $\pi_1(\Sigma)$.
When we consider the collection of lifts of $F$, the set of their conjugate actions is indeed an outer automorphism of $\pi_1(\Sigma)$.

This method is invariant under homotopies of $\Sigma$.
If $\Phi_t : \Sigma \to \Sigma$ is a homotopy from $F$ to another self-homeomorphism $F'$ on $\Sigma$, then for every lift $\tilde{F}'$, there is a lift $\tilde\Phi_t : \mathbb{H}^2 \to \mathbb{H}^2$ of $\Phi$ which is a homotopy from a lift of $F$ to $\tilde{F}'$.
But varying time $t$ does not change the conjugate action of $\tilde{F}$ because the action of $\pi_1(\Sigma)$ on $\mathbb{H}^2$ is discrete.
As a result, the mapping class of $F$ determines an outer automorphism of $\pi_1(\Sigma)$.
\end{remark}

For a simple closed geodesic $\gamma$, recall that the dual tree $\mathcal{Y}_\gamma$ of $\gamma$ is a metric space with a metric $d_\gamma$ and an isometric action $\sigma_\gamma$. (See Definition \ref{def:dual}.)

\begin{proposition} \label{prop:dehn3}
If $\gamma$ is a simple closed geodesic on $\Sigma$ and $\tilde{T}_\gamma$ is a lift of a Dehn twist $T_\gamma$ along $\gamma$, then there is an isometry $f$ on the dual tree $\mathcal{Y}_\gamma$ of $\gamma$ such that $\sigma_\gamma( \tilde{T}_\gamma g \tilde{T}_\gamma^{-1})(v) = f \sigma_\gamma(g) f^{-1}(v)$ for all $g \in \pi_1(\Sigma)$ and $v \in \mathcal{Y}_\gamma$.
\end{proposition}

From now on, we write $\sigma_\gamma(\tilde{T}_\gamma)$ instead of the isometry $f$.

\begin{proof}
Let $r > 0$ be sufficiently small satisfying that the closed $r$-neighborhood $\overline{\mathcal{N}(\gamma)}$ of $\gamma$ is homeomorphic to an annulus.
Let $\xi: \mathbb{H}^2 \to \Sigma$ be a covering map.
And let $\Phi_\gamma = \Phi_{\gamma, r}: \mathbb{H}^2 \to \mathcal{Y}_\gamma$ be a $\pi_1(\Sigma)$-equivariant surjective continuous map satisfying Proposition \ref{prop:hyp_tree}.
Define a relation $\sim$ on the open $r$-neighborhood $\mathcal{N}_r(\gamma)$ such that for all $x, y \in \mathcal{N}(\gamma)$,
\begin{center}
$x \sim y$ if and only if $\Phi_\gamma(p) = \Phi_\gamma(q)$ for some $p \in \xi^{-1}(x)$ and $q \in \xi^{-1}(y)$.
\end{center}
Then $\sim$ is an equivalence relation on $\mathcal{N}_r(\gamma)$ because $\Phi_\gamma$ is $\pi_1(\Sigma)$-equivariant.
Note that the decomposition of $\mathcal{N}_r(\gamma)$ by $\sim$ is a foliation of circles on $\mathcal{N}_r(\gamma)$ by Proposition \ref{prop:hyp_tree}\eqref{enum:sur1}.

Passing to homotopy, suppose that $T_\gamma$ supports $\mathcal{N}_r(\gamma)$ and preserves the equivalence relation $\sim$.
Then for all points $p, q \in \mathbb{H}^2$ satisfying that $\Phi_\gamma(p) = \Phi_\gamma(q)$, we have $\Phi_\gamma \tilde{T}_\gamma(p) = \Phi_\gamma \tilde{T}(q)$.
It implies that the map $f := \Phi_\gamma \tilde{T}_\gamma \Phi_\gamma^{-1}: \mathcal{Y}_\gamma \to \mathcal{Y}_\gamma$ is well-defined.
Since $\tilde{T}_\gamma$ preserves the separability of lifts of $\gamma$ (i.e., for all lifts $\tilde\gamma_1, \tilde\gamma_2, \tilde\gamma_3$ of $\gamma$, we have $\tilde\gamma_1$ separates $\tilde\gamma_2$ from $\tilde\gamma_3$ if and only if $\tilde{T}_\gamma(\tilde\gamma_1)$ separates $\tilde{T}_\gamma(\tilde\gamma_2)$ from $\tilde{T}_\gamma(\tilde\gamma_3)$) and the equivalence relation $\sim$, we have $f$ is an isometry on $\mathcal{Y}_\gamma$.

We claim that for every $g \in \pi_1(\Sigma)$ and $v \in \mathcal{Y}_\gamma$, we have $\sigma_\gamma(\tilde{T}_\gamma g \tilde{T}_\gamma^{-1}) v = f \sigma_\gamma(g) f^{-1}v$.
Since $gp \in \Phi_\gamma^{-1} \sigma_\gamma(g) \Phi_\gamma (p)$ for every point $p \in \mathbb{H}^2$, it holds that $\tilde{T}_\gamma g \tilde{T}_\gamma^{-1} (p)$ is an element of the set $\tilde{T}_\gamma ( \Phi_\gamma^{-1} \sigma_\gamma(g) \Phi_\gamma ) \tilde{T}_\gamma^{-1} (\Phi_\gamma^{-1} \Phi_\gamma) (p)$.
The $\Phi_\gamma$-image of the set $\tilde{T}_\gamma ( \Phi_\gamma^{-1} \sigma_\gamma(g) \Phi_\gamma ) \tilde{T}_\gamma^{-1} (\Phi_\gamma^{-1} \Phi_\gamma) (p)$ is a point on $\mathcal{Y}_\gamma$ and is equal to $\Phi_\gamma ((\tilde{T}_\gamma g \tilde{T}_\gamma^{-1}) p)$.
Because $\Phi_\gamma$ is $\pi_1(\Sigma)$-equivariant, we have $\sigma_\gamma(\tilde{T}_\gamma g \tilde{T}_\gamma^{-1}) \Phi_\gamma (p) = \Phi_\gamma (\tilde{T}_\gamma g \tilde{T}_\gamma^{-1}) (p)  = f \sigma_\gamma(g) f^{-1} \Phi_\gamma (p)$ for all $p \in \mathbb{H}^2$.
Because $\Phi_\gamma$ is surjective, we prove the claim.
\end{proof}

\begin{remark} \label{rem:ann_twi}
Let us remind the definition of a Dehn twist.
Let $\gamma$ be a simple closed geodesic of $\Sigma$.
For the annulus $A := (\mathbb{R} / \mathbb{Z}) \times [0, 1]$, let $\iota: A \to \Sigma$ be a topological embedding such that $\iota((\mathbb{R} / \mathbb{Z}) \times \{1/2\}) = \gamma$.
Define a map $T: A \to A$ by $T([s], t) = ([s+t], t)$ for every $([s], t) \in (\mathbb{R} / \mathbb{Z}) \times [0, 1]$.
Then the map $T_\gamma: \Sigma \to \Sigma$ defined by
$$T_\gamma(x) = 
\begin{cases} 
(\iota \circ T \circ \iota^{-1}) (x) & \text{if } x \in \iota(A), \\ 
x, & \text{otherwise}, 
\end{cases}$$
is a Dehn twist.
(See Farb-Margalit's definition \cite[Section 3.1.1]{fm12}.)

Note that $\pi_1(A) = \langle a \rangle$ acts on $\tilde{A} = \mathbb{R} \times [0, 1]$ by $a^m(s, t) = (s+m, t)$ for all $m \in \mathbb{Z}$ and $(s, t) \in \tilde{A}$.
If $\tilde{T}: \tilde{A} \to \tilde{A}$ is the lift of $T$ fixing $\mathbb{R} \times \{ 0 \}$ pointwise, the equation
$$\tilde{T}^m(s, 1) = (s + m, 1) = a^m(s, 1)$$
holds for all $m \in \mathbb{Z}$ and $s \in \mathbb{R}$.
We want to apply this equation to lifts of a Dehn twist.
Every lift $\tilde\iota: \tilde{A} \to \mathbb{H}^2$ of $\iota$ gives an injection $\tilde\iota_*: \pi_1(A) \to \pi_1(\Sigma)$.
If $\tilde{T}_\gamma$ is the lift of $T_\gamma$ which fixes a side of $\tilde\iota(\tilde{A})$ pointwise, then there is a primitive element $h \in \{ \tilde\iota_*(a), \tilde\iota_*(a^{-1}) \}$ such that we have $$\tilde{T}_\gamma^m(p) = h^m p$$  for every point $p$ on the other side of $\tilde\iota(\tilde{A})$ and $m \in \mathbb{Z}$.
\end{remark}

\begin{proposition} \label{prop:dehn_cal}
Let $\gamma$ be a simple closed geodesic on $\Sigma$, and let $v$ be a vertex of the dual tree $\mathcal{Y}_\gamma$.
Then there is a lift $\tilde{T}_\gamma$ of a Dehn twist $T_\gamma$ such that $\sigma_\gamma(\tilde{T}_\gamma)$ fixes the closed $1$-neighborhood of $v$ pointwise.

Furthermore, if $w$ is a vertex of $\mathcal{Y}_\gamma$ distinct from $v$ and $\vec{L}:[0, n] \to \mathcal{Y}_\gamma$ is the unit-speed geodesic path from $v$ to $w$, then there are primitive elements $h_1, \dots, h_n \in \pi_1(\Sigma)$ such that $h_i$ fixes the midpoint $\vec{L}(i - 1/2)$ for each $i = 1, \dots, n$ and such that $\sigma_\gamma(\tilde{T}_\gamma^m)w = \sigma_\gamma(h_1^m \dots h_n^m) w$ for every integer $m$; see Figure \ref{fig:dehn_cal}.
\end{proposition}
\begin{figure}
\centering
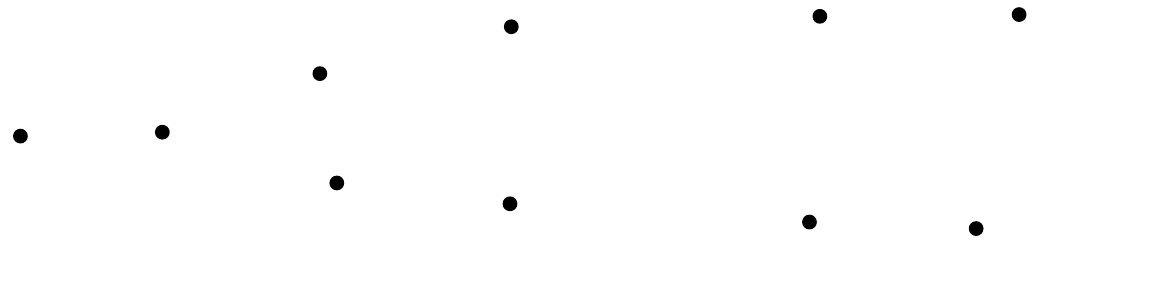
\caption{\label{fig:dehn_cal}}
\end{figure}

\begin{proof}
Let $r > 0$ be small enough that the closed $r$-neighborhood of $\gamma$ is homeomorphic to an annulus.
Let $\Phi_\gamma : = \Phi_{\gamma, r}: \mathbb{H}^2 \to \mathcal{Y}_\gamma$ be a $\pi_1(\Sigma)$-equivariant surjective continuous map satisfying Proposition \ref{prop:hyp_tree}.
Passing to homotopy, suppose that $T_\gamma$ is a Dehn twist along $\gamma$ such that the support of $T_\gamma(\gamma)$ is the $r$-neighborhood of $\gamma$.
Let $p$ be a point on $\mathbb{H}^2$ such that $\Phi_\gamma(p) = v$.

Because the projective image of $p$ on $\Sigma$ is not contained in the open support of $T_\gamma$, there is a lift $\tilde{T}_\gamma$ of $T_\gamma$ which fixes $p$.
For every edge $e$ incident to $v$, if $\tilde\gamma$ is the lift of $\gamma$ such that $\Phi_\gamma(\tilde\gamma)$ is the midpoint of $e$, then $\tilde{T}_\gamma(\tilde\gamma) = \tilde\gamma$ because some connected component of the boundary of the $r$-neighborhood of $\tilde\gamma$ is fixed pointwise by $\tilde{T}_\gamma$.
So $\sigma_\gamma(\tilde{T}_\gamma)e = e$.
Therefore, $\tilde{T}_\gamma$ fixes the closed $1$-neighborhood of $v$ pointwise.

To prove the second statement in Proposition \ref{prop:dehn_cal}, we use an induction on the distance $n$ between $v$ and $w$.
If $v$ and $w$ are joined by an edge $e$ and a primitive element $h$ of $\pi_1(\Sigma)$ fixes the edge $e$, then we have $\sigma_\gamma(\tilde{T}_\gamma^m)w = w = \sigma_\gamma(h^m) w$ for every integer $m$ by the above.

Assume that $n = d_\gamma(v, w) \geq 2$.
Let an integer $m$ be given.
Let $\tilde{T}_\gamma'$ be the lift of $T_\gamma$ fixing the closed $1$-neighborhood of $\vec{L}(1)$.
By the induction hypothesis, there are primitive elements $h_2, \dots, h_n \in \pi_1(\Sigma)$ fixing midpoints $\vec{L}(3/2), \dots, \vec{L}(n- 1/2)$, respectively, such that $\sigma_\gamma((\tilde{T}_\gamma')^m)(w) = \sigma_\gamma(h_2^m \dots h_n^m) w$.

We claim that $(\tilde{T}_\gamma')^m = h_1^{-m} \tilde{T}_\gamma^m$ for some primitive element $h_1$ fixing $\vec{L}(1/2)$.
Let $\tilde\gamma$ be the lift of $\gamma$ such that $\Phi_\gamma(\tilde\gamma) = \vec{L}(1/2)$.
Then $\tilde{T}_\gamma'$ fixes some boundary component $\partial_+\mathcal{N}_r(\tilde\gamma)$ of the closed $r$-neighborhood of $\tilde\gamma$ pointwise.
By Remark \ref{rem:ann_twi}, there is a primitive element $h_1 \in \pi_1(\Sigma)$ preserving $\tilde\gamma$ such that $\tilde{T}_\gamma^m(p) = h_1^m p$ for every $p \in \partial_+\mathcal{N}_r(\tilde\gamma)$.
Because $h_1^{-m} \tilde{T}_\gamma^m$ fixes $p$ and is a lift of $T_\gamma^m$, we have $h_1^{-m} \tilde{T}_\gamma^m = (\tilde{T}_\gamma')^m$ by the path lifting property.
Therefore, $\sigma_\gamma(\tilde{T}_\gamma^m)w = \sigma_\gamma(h_1^m h_2^m \dots h_n^m) w$.
\end{proof}

\section{The Proof of Theorem \ref{thm:pp_main}} \label{sec:fpc}

Let us prove Theorem \ref{thm:pp_main}.
First, we will introduce the general definition of $\operatorname{PP}_n$. (\emph{cf.} Definition \ref{def:pp_simple})
We recall that if $\mathcal{F}$ is a set of simple closed geodesics on $\Sigma$, a lift of $\mathcal{F}$ means a lift of a simple closed geodesic in $\mathcal{F}$.
And recall that a multicurve is a set of pairwise disjoint simple closed geodesics on $\Sigma$.

\begin{definition}[General definition of $\operatorname{PP}_n$] \label{def:ppp}
Let $\mathcal{A}$ and $\mathcal{B}$ be multicurves on $\Sigma$, and let $\alpha$ and $\beta$ be simple closed geodesics in $\mathcal{A}$ and $\mathcal{B}$, respectively.
For $n > 2$, a simple closed geodesic $\gamma$ of $\Sigma$ is said to be contained in  $\operatorname{PP}_n(\mathcal{A}, \alpha, \mathcal{B}, \beta)$ if for some lift $\tilde\gamma$ of $\gamma$, for some lift $\tilde\alpha$ of $\alpha$ and (at least) $n$ lifts $\tilde\beta_1, \dots, \tilde\beta_n$ of $\mathcal{B}$ such that the following hold.

\begin{enumerate}[label=(\roman*)]
\item \label{enum:fp-2}
Each $\tilde\beta_i$ separates $\tilde\beta_{i-1}$ from $\tilde\beta_{i+1}$.
\item \label{enum:fp1}
There is an index $i_0 \in \{2, \dots, n-1\}$ such that $\tilde\beta_{i_0}$ is a lift of $\beta$.
\item \label{enum:fp2}
For all $i = 1, \dots, n$, the triple $\tilde\gamma$, $\tilde\alpha$ and $\tilde\beta_i$ cross each other.
\item \label{enum:fp3}
If a lift $\tilde\alpha'$ of $\mathcal{A}$ intersects $\tilde\gamma$ and some $\tilde\beta_i$, then $\tilde\alpha' = \tilde\alpha$.
\end{enumerate}
\end{definition}

The set $\operatorname{PP}_n(\mathcal{A}, \alpha, \beta)$ in Definition \ref{def:pp_simple} is equal to $\operatorname{PP}_n(\mathcal{A}, \alpha, \{\beta\}, \beta)$ with $\alpha \pitchfork \beta$.
We collect basic properties of $\operatorname{PP}_n$ in the next lemma.

\begin{lemma}\label{lem:admi}
Let $\mathcal{A}$ and $\mathcal{B}$ be multicurves, and let simple closed geodesics $\alpha \in \mathcal{A}$ and $\beta \in \mathcal{B}$ be given.
For $n > 2$, the following hold.
\begin{enumerate}
\item
\label{enum:ad_non}
If $\operatorname{PP}_n(\mathcal{A}, \alpha, \mathcal{B}, \beta)$ is nonempty, then $i(\alpha, \beta) > 0$.
\item
\label{enum:ad_len}
If $2 < m \leq n$, then $\operatorname{PP}_m(\mathcal{A}, \alpha, \mathcal{B}, \beta) \supseteq \operatorname{PP}_n(\mathcal{A}, \alpha, \mathcal{B}, \beta)$.
\item
\label{enum:ad_for}
If $\mathcal{A}_1$ and $\mathcal{A}_2$ are multicurves satisfying that $\alpha \in \mathcal{A}_1 \subseteq \mathcal{A}_2$, then $$\operatorname{PP}_n(\mathcal{A}_1, \alpha, \mathcal{B}, \beta) \supseteq \operatorname{PP}_n(\mathcal{A}_2, \alpha, \mathcal{B}, \beta).$$
\item
\label{enum:ad_lat}
If $\mathcal{B}_1$ and $\mathcal{B}_2$ are multicurves satisfying that $\beta \in \mathcal{B}_1 \subseteq \mathcal{B}_2$, then $$\operatorname{PP}_n(\mathcal{A}, \alpha, \mathcal{B}_1, \beta) \subseteq \operatorname{PP}_n(\mathcal{A}, \alpha, \mathcal{B}_2, \beta).$$
\end{enumerate}
\end{lemma}

\begin{proof}
\noindent\eqref{enum:ad_non}
If some simple closed geodesic is contained in $\operatorname{PP}_n( \mathcal{A}, \alpha, \mathcal{B}, \beta)$, then $\tilde\alpha \pitchfork \tilde\beta$ for some lift $\tilde\alpha$ of $\alpha$ and lift $\tilde\beta$ of $\beta$ by Definition \ref{def:ppp}\ref{enum:fp1} and \ref{enum:fp2}.
It implies that $i(\alpha,\beta) > 0$. \vspace{1.5mm}

\noindent\eqref{enum:ad_len}
Choose a simple closed geodesic $\gamma \in \operatorname{PP}_n(\mathcal{A}, \alpha, \mathcal{B}, \beta)$.
Then there are a lift $\tilde\gamma$ of $\gamma$, a lift $\tilde\alpha$ of $\alpha$ and $n$ lifts $\tilde\beta_1, \dots, \tilde\beta_n$ of $\mathcal{B}$ satisfying \ref{enum:fp-2} to \ref{enum:fp3} of Definition \ref{def:ppp}.
If $\tilde\beta_{i_0}$ (for some $1 < i_0 < n$) is a lift of $\beta$, choose a subsequence $\tilde\beta_{j_1}, \dots, \tilde\beta_{j_m}$ of length $m$ from $\{\tilde\beta_1, \dots, \tilde\beta_n\}$ such that $\tilde\beta_1$, $\tilde\beta_{i_0}$ and $\tilde\beta_n$ are contained in the subsequence.
Then $\tilde\gamma$, $\tilde\alpha$ and the subsequence $\tilde\beta_{j_1}, \dots, \tilde\beta_{j_m}$ satisfy \ref{enum:fp-2} to \ref{enum:fp3} of Definition \ref{def:ppp} of $\operatorname{PP}_m$ for $\gamma$.
That is, $\gamma \in \operatorname{PP}_m(\mathcal{A}, \alpha, \mathcal{B}, \beta)$.
\vspace{1.5mm}

\noindent\eqref{enum:ad_for}
Let a simple closed geodesic $\gamma \in \operatorname{PP}_n(\mathcal{A}_2, \alpha, \mathcal{B}, \beta)$ be given.
Then there are a lift $\tilde\gamma$ of $\gamma$, a lift $\tilde\alpha$ of $\alpha$ and $n$ lifts $\tilde\beta_1 \dots \tilde\beta_n$ of $\mathcal{B}$ satisfying \ref{enum:fp-2} to \ref{enum:fp3} of Definition \ref{def:ppp}.
If a lift $\tilde\alpha'$ of $\mathcal{A}_1$ intersects both $\tilde\gamma$ and some $\tilde\beta_i$, then $\tilde\alpha' = \tilde\alpha$ because $\tilde\alpha'$ is a lift of $\mathcal{A}_2$.
So \ref{enum:fp3} of Definition \ref{def:ppp} holds.
Note that the conditions \ref{enum:fp-2} to \ref{enum:fp2} of Definition \ref{def:ppp} are satisfied immediately with $\tilde\gamma$, $\tilde\alpha$ and $\{ \tilde\beta_1, \dots, \tilde\beta_n \}$.
Therefore, $\gamma \in \operatorname{PP}_n(\mathcal{A}_1, \alpha, \mathcal{B}, \beta)$.
\vspace{1.5mm}

We leave \eqref{enum:ad_lat} as an exercise.
\end{proof}

\begin{remark}
The below proposition (Proposition \ref{prop:half_ele}) induces the converse of Lemma \ref{lem:admi}\eqref{enum:ad_non} so the following statement is satisfied. \emph{Whenever $\mathcal{A}$ and $\mathcal{B}$ are multicurves and $\alpha$ and $\beta$ are simple closed geodesics in $\mathcal{A}$ and $\mathcal{B}$, respectively, we have $i(\alpha, \beta) > 0$ if and only if $\operatorname{PP}_n(\mathcal{A}, \alpha, \mathcal{B}, \beta) \neq \emptyset$ for all $n > 2$.}
\end{remark}

\begin{lemma} \label{lem:no_inc}
Let $\alpha$ and $\beta$ be simple closed geodesics crossing each other.
If  a simple closed geodesic $\gamma$ is contained in $\operatorname{PP}_n(\{\alpha\}, \alpha, \beta) = \operatorname{PP}_n(\{\alpha\}, \alpha, \{\beta\}, \beta)$ for some $n$, then $n \leq \operatorname{lcm}\{i(\alpha, \beta), i(\beta, \gamma)\}$.
\end{lemma}

\begin{proof}
Assume that $\gamma \in \operatorname{PP}_n(\{\alpha\}, \alpha, \beta)$ for some $n > 2$.
Then there are a lift $\tilde\gamma$ of $\gamma$, a lift $\tilde\alpha$ of $\alpha$ and $n$ lifts $\tilde\beta_1, \dots, \tilde\beta_n$ of $\beta$ such that $\tilde\alpha$, $\tilde\beta_i$ and $\tilde\gamma$ cross each other for each $i$ by \ref{enum:fp2} of Definition \ref{def:ppp}.
Then $n \leq \operatorname{lcm}\{ i(\alpha, \beta), i(\beta, \gamma) \}$ by Proposition \ref{prop:num_lift}\eqref{enum:cro_num}.
\end{proof}

\begin{proposition} \label{prop:half_ele}
Let $\alpha$, $\beta$ and $\gamma$ be simple closed geodesics such that $i(\alpha, \beta) > 0$ and $i(\beta, \gamma) > 0$.
Let $n > 2$ be given.
If an integer $m$ satisfies that $$\lvert m \rvert \geq \frac{n + 2 \cdot i(\alpha, \gamma)}{i(\beta, \gamma)} + 2,$$ then for every multicurve $\mathcal{B}$ containing $\beta$,
$$T_\beta^m \left( \{\alpha\} \cup \operatorname{PP}_3(\{ \alpha \}, \alpha, \mathcal{B}, \beta) \right) \subseteq \operatorname{PP}_n(\mathcal{B}, \beta, \{ \gamma \}, \gamma).$$
\end{proposition}

\begin{proof}
Fix a multicurve $\mathcal{B}$ containing $\beta$.
And choose $\delta \in \{\alpha\} \cup \operatorname{PP}_3(\{ \alpha \}, \alpha, \mathcal{B}, \beta)$.
Let $\tilde\delta$ and $\tilde\alpha$ be lifts of $\delta$ and $\alpha$, respectively, and let $\tilde\beta_{-1}$, $\tilde\beta_0$ and $\tilde\beta_1$ be lifts of $\mathcal{B}$ satisfying the following; see Figure \ref{fig:pp_H}.

\begin{itemize}
\item If $\delta = \alpha$, let $\tilde\delta = \tilde\alpha$ be an arbitrary lift of $\alpha$.
Let $\tilde\beta_0$ be a lift of $\beta$ crossing $\tilde\alpha$.
Let $\tilde\beta_{-1}$ and $\tilde\beta_1$ be lifts of $\mathcal{B}$ crossing $\tilde\alpha$ such that $\tilde\beta_0$ is the unique lift of $\mathcal{B}$ separating $\tilde\beta_{-1}$ from $\tilde\beta_1$.

\item If $\delta \in PP_3(\{ \alpha \}, \alpha, \mathcal{B}, \beta)$, then let $\tilde\delta$ and $\tilde\alpha$ be lifts of $\delta$ and $\alpha$, respectively, and let $\tilde\beta_{-1}, \tilde\beta_0, \tilde\beta_1$ be lifts of $\mathcal{B}$ satisfying \ref{enum:fp1} to \ref{enum:fp3} of Definition \ref{def:ppp} and satisfying that $\tilde\beta_0$ separates $\tilde\beta_{-1}$ from $\tilde\beta_1$.
If $\tilde\beta_0$ and some $\tilde\beta_i$ are separated by a lift $\tilde\beta'$ of $\mathcal{B}$, then substitute $\tilde\beta_i$ to $\tilde\beta'$.
After then, the lifts $\tilde\delta, \tilde\alpha, \tilde\beta_{-1}, \tilde\beta_0, \tilde\beta_1$ also satisfy Definition \ref{def:ppp} for $\delta \in PP_3(\{ \alpha \}, \alpha, \mathcal{B}, \beta)$.
Repeating this process, we obtain the fact that $\tilde\beta_0$ is the unique lift of $\mathcal{B}$ separating $\tilde\beta_{-1}$ from $\tilde\beta_1$.
Note that $\tilde\beta_0$ is a lift of $\beta$ by \ref{enum:fp1} of Definition \ref{def:ppp}.
\end{itemize}

\begin{figure}
\centering
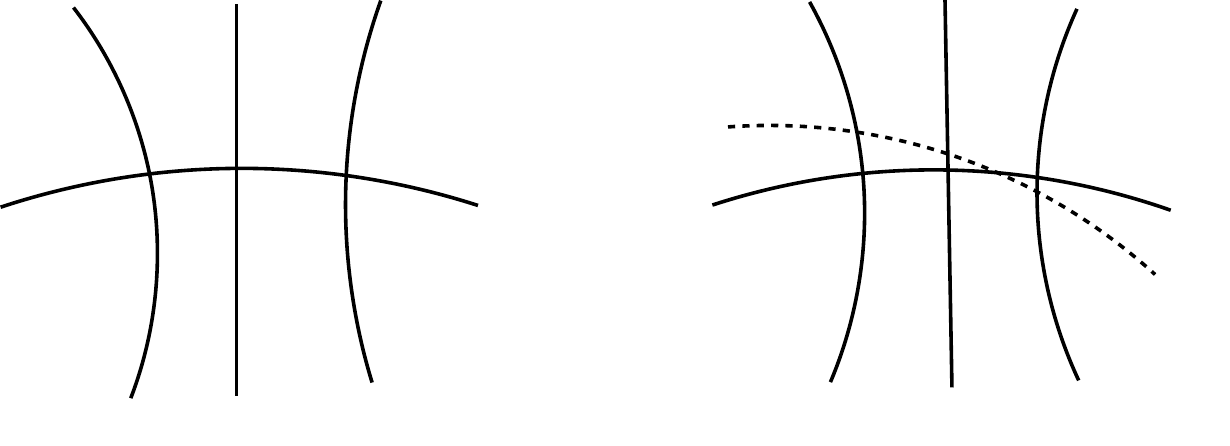
\caption{}
\label{fig:pp_H}
\end{figure}

\vspace{2mm}

\noindent\textbf{\textsf{\underline{Claim. The number of lifts of $\gamma$ separating $\tilde\beta_{-1}$ from $\tilde\beta_1$ is at most $2 \cdot i(\alpha, \gamma)$.}}}
\vspace{1mm}

Let $h$ be a primitive element of $\pi_1(\Sigma)$ preserving $\tilde\alpha$ such that $h\tilde\beta_{-1}$ lies in the connected component of $\mathbb{H}^2 \setminus \tilde\beta_{-1}$ containing $\tilde\beta_0$.
Because $h \tilde\beta_{-1}$ cannot separate $\tilde\beta_0$ from $\tilde\beta_{-1}$, either $h \tilde\beta_{-1} = \tilde\beta_0$ or $\tilde\beta_0$ separates $\tilde\beta_{-1}$ from $h \tilde\beta_{-1}$.
Since $h$ sends $\tilde\beta_0$ into the connected component of $\mathbb{H}^2 \setminus \tilde\beta_0$ containing $\tilde\beta_1$, either $h \tilde\beta_0 = \tilde\beta_1$ or $\tilde\beta_1$ separates $\tilde\beta_0$ from $h \tilde\beta_0$.
Combining these facts, we obtain that either $h^2 \tilde\beta_{-1} = \tilde\beta_1$ or $\tilde\beta_1$ separates $\tilde\beta_{-1}$ from $h^2 \tilde\beta_{-1}$.

Let $\mathcal{Y}_\gamma$ be the dual tree of $\gamma$.
Let $\Phi_\gamma := \Phi_{\gamma, r}: \mathbb{H}^2 \to \mathcal{Y}_\gamma$ be a $\pi_1(\Sigma)$-equivariant surjective continuous map satisfying Proposition \ref{prop:hyp_tree} for some sufficiently small number $r > 0$.
For each $i = -1, 0, 1$, it holds that $\Phi_\gamma(\tilde\alpha)$ intersects $\Phi_\gamma(\tilde\beta_i)$ because $\tilde\alpha \pitchfork \tilde\beta_i$.
See Figure \ref{fig:geods_tree}.
Every path joining $\Phi_\gamma(\tilde\beta_{-1})$ and $ \sigma_\gamma(h^2) \Phi_\gamma(\tilde\beta_{-1})$ intersects $\Phi_\gamma(\tilde\beta_1)$ by the above.
So we have $d_\gamma(\Phi_\gamma(\tilde\beta_{-1}), \Phi_\gamma(\tilde\beta_1)) \leq d_\gamma(\Phi_\gamma(\tilde\beta_{-1}), \sigma_\gamma(h^2) \Phi_\gamma(\tilde\beta_{-1}))$.

Because $\mathcal{Y}_\gamma$ is a tree, the shortest geodesic between $\Phi_\gamma(\tilde\beta_{-1})$ and $\Phi_\gamma(\tilde\beta_1)$ is contained in $\Phi_\gamma(\tilde\alpha)$.
Since the translation length of $h$ on $\mathcal{Y}_\gamma$ is $i(\alpha, \gamma)$, it is satisfied that $$d_\gamma(\Phi_\gamma(\tilde\beta_{-1}), \Phi_\gamma(\tilde\beta_1)) \leq 2 \cdot i(\alpha, \gamma).$$
So the claim holds.
\vspace{2mm}

\begin{figure}
\centering
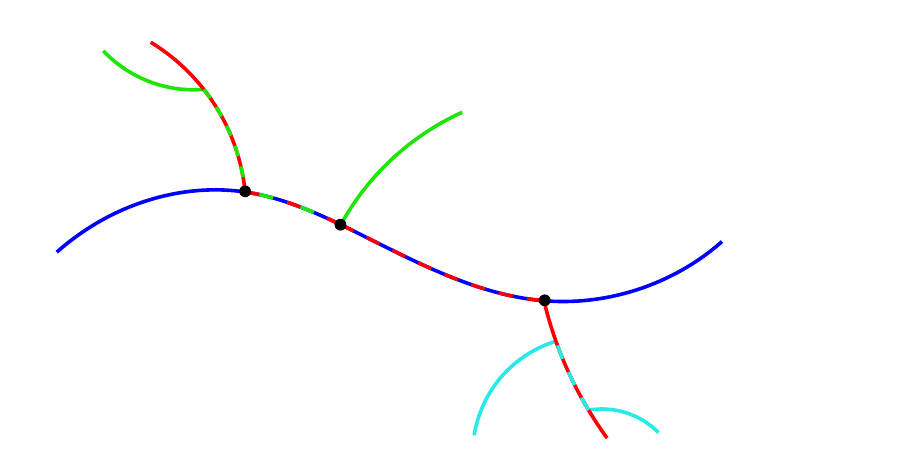
\caption{}
\label{fig:geods_tree}
\end{figure}

Let $v$ be the vertex on the dual tree $\mathcal{Y}_\beta$ of $\beta$, which lies on the geodesic segment connecting $\Phi_\beta(\tilde\beta_{-1})$ and $\Phi_\beta(\tilde\beta_0)$.
Let $\tilde{T}_\beta$ be the lift of a Dehn twist $T_\beta$ which fixes the closed $1$-neighborhood of $v$. (\emph{cf.} Proposition \ref{prop:dehn_cal})
Because $\Phi_\beta(\tilde\beta_{-1})$ and $\Phi_\beta(\tilde\beta_0)$ are contained in the closed $1$-neighborhood of $v$, we have $\tilde{T}_\beta^m \tilde\beta_{-1} = \tilde\beta_{-1}$ and $\tilde{T}_\beta^m \tilde\beta_0 = \tilde\beta_0$.
On the other hand, since $\tilde\beta_0$ is the unique lift of $\beta$ separating $\tilde\beta_1$ from $\Phi_\beta^{-1}(v)$, we have $\tilde{T}_\beta^m \tilde\beta_1 = h_0^m \tilde\beta_1$ for some primitive element $h_0$ preserving $\tilde\beta_0$ by Proposition \ref{prop:dehn_cal}.
\vspace{2mm}

\noindent\textbf{\textsf{\underline{Claim. The number of lifts of $\gamma$ separating $\tilde\beta_{-1}$ from $\tilde{T}_\beta^m\tilde\beta_1$ is at most $n$.}}}
\vspace{1mm}

If $\mathrm{proj}: \mathcal{Y}_\gamma \to \Phi_\gamma(\tilde\beta_0)$ be the closest-point projection to $\Phi_\gamma(\tilde\beta_0)$, let $P_i$ denote $\mathrm{proj}(\Phi_\gamma(\tilde\beta_i))$ for each $i = -1, 1$.
Then the length $l(P_i)$ of each $P_i$ is at most $i(\beta, \gamma)$ by Proposition \ref{prop:num_lift}.
Let $w'$ and $w''$ be the vertices in $P_1$ satisfying that $d_\gamma(w', \sigma_\gamma(h_0^m) w'') = d_\gamma(P_1, \sigma_\gamma(h_0^m) P_1)$.
Since $\Phi_\gamma(\tilde\beta_0)$ is preserved by $h_0$,
\begin{align*}
d_\gamma(P_1, \sigma_\gamma(h_0^m) P_1) &= d_\gamma(w', \sigma_\gamma(h_0^m) w'') \\
& \geq d_\gamma(w', \sigma_\gamma(h_0^m) w') - d_\gamma(\sigma_\gamma(h_0^m) w', \sigma_\gamma(h_0^m) w'') \\
& \geq \lvert m \rvert \cdot \mathrm{tr}_\gamma\, h_0 - l(P_1) \\
& \geq (\lvert m \rvert - 1) \cdot i(\beta, \gamma).
\end{align*}
By the hypothesis and the previous claim,
\begin{align*}
d_\gamma(\Phi_\gamma(\tilde\beta_{-1}), \sigma_\gamma(h_0^m) \Phi_\gamma(\tilde\beta_1)) & \geq d_\gamma(P_{-1}, \sigma_\gamma(h_0^m) P_1) \\
& \geq d_\gamma(P_1, \sigma_\gamma(h_0^m) P_1) - d_\gamma(P_{-1}, P_1) - l(P_{-1}) \\
&\geq (\lvert m \rvert - 1) \cdot i(\beta, \gamma) - 2 \cdot i(\alpha, \gamma) - i(\beta, \gamma) \\
&\geq n.
\end{align*}
So there are at least $n$ edges of $\mathcal{Y}_\gamma$ between $\Phi_\gamma(\tilde\beta_{-1})$ and $\sigma_\gamma(h_0^m) \Phi_\gamma(\tilde\beta_1)$.
In other words, there are $n$ lifts $\tilde\gamma_1, \dots, \tilde\gamma_n$ of $\gamma$ separating $\tilde\beta_{-1}$ from $h_0^m \tilde\beta_1$.
We proved the claim.
\vspace{2mm}

\begin{figure}
\centering
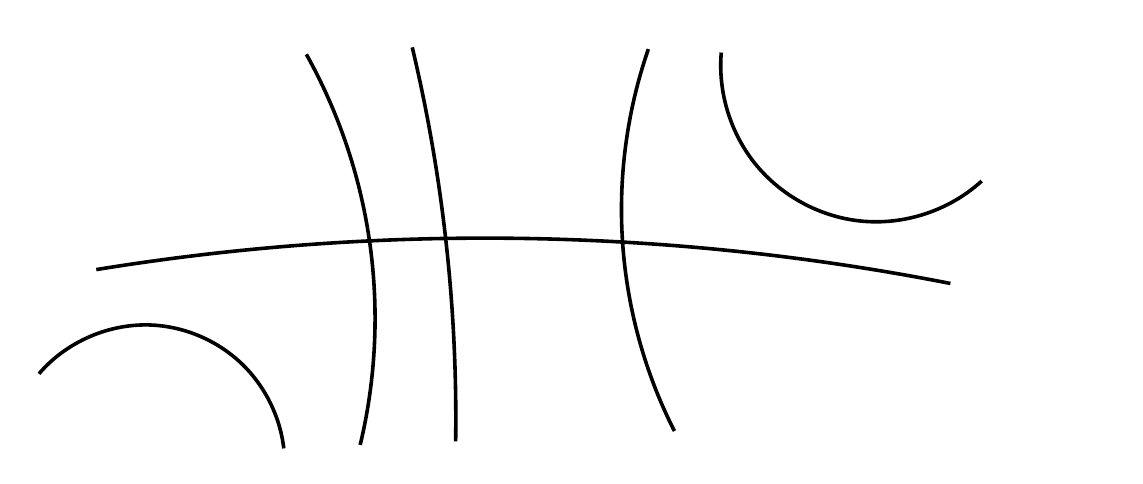
\caption{}
\label{fig:pp_HC}
\end{figure}

Since $\tilde{T}_\beta^m \tilde\delta$ crosses both $\tilde{T}_\beta^m \tilde\beta_{-1} = \tilde\beta_{-1}$ and $\tilde{T}_\beta^m \tilde\beta_1$, it also crosses all $\tilde\gamma_1, \dots, \tilde\gamma_n$.
See Figure \ref{fig:pp_HC}.
Note that $\tilde{T}_\beta^m\tilde\beta_0$ is the unique lift of $\mathcal{B}$ separating $\tilde\beta_{-1}$ from $\tilde{T}_\beta^m\tilde\beta_1$.
If a lift $\tilde\beta''$ of $\mathcal{B}$ intersects $\tilde{T}_\beta^m\tilde\delta$ and some $\tilde\gamma_i$, then it must separates $\tilde\beta_{-1}$ from $h_0^m\tilde\beta_1$ so that $\tilde\beta'' = \tilde\beta_0$.
Hence, $T_\beta^m\delta$ is contained in $\operatorname{PP}_n(\mathcal{B}, \beta, \{ \gamma \}, \gamma)$.
\end{proof}

\begin{proposition} \label{prop:fp_con}
Let $\mathcal{A}$ and $\mathcal{B}$ be multicurves, and let $\alpha \in \mathcal{A}$ and $\beta \in \mathcal{B}$ be simple closed geodesics such that $i(\alpha, \beta) > 0$.
Then
$$T_\gamma^m ( \operatorname{PP}_n(\mathcal{A}, \alpha, \mathcal{B}, \beta) ) \subset \operatorname{PP}_n(\mathcal{A}, \alpha, \mathcal{B}, \beta)$$
for all $\gamma \in \mathcal{A} - \{ \alpha \}$, $n > 2$ and $m \in \mathbb{Z}$.
\end{proposition}

\begin{proof}
Choose $\delta \in \operatorname{PP}_n(\mathcal{A}, \alpha, \mathcal{B}, \beta)$.
Let $\tilde\delta$ and $\tilde\alpha$ be lifts of $\delta$ and $\alpha$, respectively, and let $\tilde\beta_1, \dots, \tilde\beta_n$ be lifts of $\mathcal{B}$ satisfying \ref{enum:fp-2} to \ref{enum:fp3} of Definition \ref{def:ppp} for $\delta \in \operatorname{PP}_n(\mathcal{A}, \alpha, \mathcal{B}, \beta)$.
Let $\tilde\alpha_{-1}$ and $\tilde\alpha_1$ be lifts of $\mathcal{A}$ such that $\tilde\alpha$ is the unique lift of $\mathcal{A}$ separating $\tilde\alpha_{-1}$ from $\tilde\alpha_1$.
Because $\gamma$ is disjoint from $\alpha$, there is a lift $\tilde{T}_\gamma$ of a Dehn twist $T_\gamma$ whose support is disjoint from $\tilde\alpha$.

In this case, we have $\tilde{T}_\gamma^m \tilde\alpha_i = \tilde\alpha_i$ for all $i = -1, 1$.
So $\tilde{T}_\gamma^m \tilde\delta$ still crosses both $\tilde\alpha_{-1}$ and $\tilde\alpha_1$.
It implies that all $\tilde\beta_1, \dots, \tilde\beta_n$ cross $\tilde{T}_\gamma^m \tilde\delta$.
Because $\tilde\alpha$ is the unique lift of $\mathcal{A}$ separating $\tilde\alpha_{-1}$ from $\tilde\alpha_1$, \ref{enum:fp3} of Definition \ref{def:ppp} also holds for $\tilde{T}_\gamma \tilde\beta$.
Therefore, $T_\gamma^m \delta \in \operatorname{PP}_n(\mathcal{A}, \alpha, \mathcal{B}, \beta)$.
\end{proof}

\begin{proof}[Proof of Theorem \ref{thm:pp_main}] \label{pf:pp_main}
\noindent \eqref{enum:pmn1} and \eqref{enum:pmn3} are direct consequences of Lemma \ref{lem:no_inc} and Proposition \ref{prop:fp_con}.
\vspace{1mm}

\noindent \eqref{enum:pmn2} By Lemma \ref{lem:admi}\eqref{enum:ad_len}, \eqref{enum:ad_for} and \eqref{enum:ad_lat}, we have $$\operatorname{PP}_n(\mathcal{A}, \alpha, \beta) \subseteq \operatorname{PP}_3(\{ \alpha \}, \alpha, \beta) \subseteq \operatorname{PP}_3(\{ \alpha \}, \alpha, \mathcal{B}, \beta).$$
Therefore, $T_\beta^m(\{ \alpha \} \cup \operatorname{PP}_n(\mathcal{A}, \alpha, \beta)) \subseteq \operatorname{PP}_n(\mathcal{B}, \beta, \gamma)$ by Proposition \ref{prop:half_ele}.
\end{proof}

\bibliographystyle{amsalpha}
\bibliography{references}

\end{document}